\documentclass{daj}
\usepackage{mathrsfs}
\usepackage{amssymb,latexsym,amsmath,amsthm,enumitem,mathrsfs}
\usepackage{color}
\usepackage{stmaryrd}
\usepackage{mathrsfs}
\usepackage{lineno}
\usepackage{bbm}
\usepackage{scalerel,stackengine}
\usepackage{tikz}
\usepackage{enumitem}
\usepackage{ esint }
\usepackage{graphicx}
\usepackage{amssymb}

\newtheorem{theorem}{Theorem}[section]
\newtheorem{conjecture}[theorem]{Conjecture}
\newtheorem{lemma}[theorem]{Lemma}

\newtheorem{remark}[theorem]{Remark}
\newtheorem{proposition}[theorem]{Proposition}
\newtheorem{question}{Question}
\newtheorem{definition}[theorem]{Definition}

\theoremstyle{definition}

\newcommand{\cL}{\mathcal{L}}

\newcommand{\cU}{\mathcal{U}}
\newcommand{\cV}{\mathcal{V}}
\newcommand{\cW}{\mathcal{W}}

\newcommand{\bA}{\mathbf{A}}

\newcommand{\bF}{\mathbf{F}}

\newcommand{\bM}{\mathbf{M}}

\newcommand{\Z}{\mathbb Z}

\newcommand{\N}{\mathbb N}

\newcommand{\F}{\mathbb F}

\newcommand{\e}{\varepsilon}

\newcommand{\Fp}{\mathbb{F}_p}

\renewcommand{\L}{\mathbb L}

\def\bA{\mathbf A}

\def\bE{\mathbf E}

\def\set4{\mathcal I}
\def\tup14{(1,2,3,4)}

\newcommand{\R}{\mathbb{R}}

\newcommand{\De}{\Delta} 
\newcommand{\ga}{\gamma}

\newcommand{\wt}{\widetilde}
\newcommand{\si}{\sigma}

\newcommand{\bzz}{\mathbf{0}}
\newcommand{\lam}{\lambda}

\dajAUTHORdetails{%
  title = {Exceptional set estimate in prime fields: dimension two implies higher dimensions}, 
  author = {Shengwen Gan},
  plaintextauthor = {Shengwen Gan},
  plaintexttitle = {Exceptional set estimate in prime fields: dimension two implies higher dimensions}, 
  runningtitle = {Exceptional set estimate in prime fields}, 
  runningauthor = {Shengwen Gan},
  copyrightauthor = {Shengwen Gan},
  keywords = {Furstenberg set problem, Marstrand projection},
}   

\dajEDITORdetails{%
   year={2025},
   number={19},
   received={4 December 2024},   
   published={29 September 2025},  
   doi={10.19086/da.144860},       
}   

\begin{document}

\begin{frontmatter}[classification=text]

\title{Exceptional set estimate in prime fields: dimension two implies higher dimensions} 

\author[sgan]{Shengwen Gan}

\begin{abstract}
In this paper, we study in prime fields the exceptional set estimates, which can be viewed as a refinement of Marstrand's orthogonal projection theorem. Additionally, we address a Furstenberg-type problem, which is closely related.
It is shown that the two-dimensional result implies all higher-dimensional results in the prime-field setting.
\end{abstract}
\end{frontmatter}


\section{Introduction}

We study the exceptional set estimate and a Furstenberg-type problem in prime fields. The main novelty of this paper is that we prove, for these two problems, that the two-dimensional result implies all higher-dimensional results. 

\subsection{Background}

We first introduce the two problems---the Furstenberg set problem and the exceptional set estimate in Euclidean space.

\subsubsection{Furstenberg set problem}\hfill

The Furstenberg set problem was originally considered by Wolff \cite{wolff1999recent} in $\R^2$. Fix $s\in (0,1], t\in (0,2]$. We say $E\subset \R^2$ is an $(s,t)$-set in $\R^2$, if $E$ is of the form
\[ E=\bigcup_{\ell\in \cL}Y(\ell), \]
where $\cL\subset A(1,\R^2)$ is a set of lines in $\R^2$ with $\dim\cL=t$ and for each $\ell\in \cL$ there exist $Y(\ell)\subset \ell$ with $\dim Y(\ell)=s$. ($\dim$ will always mean Hausdorff dimension.) The Furstenberg set problem asks about the optimal lower bound on the Hausdorff dimension of $(s,t)$-sets, that is,
\[ \inf_{E \textup{~is~a~}(s,t)\textup{-set}}\dim(E). \]
Recently, Ren and Wang \cite{ren2023furstenberg} resolved the Furstenberg set problem by proving
\begin{equation}\label{renwang}
\inf_{E \textup{~is~a~}(s,t)\textup{-set}}\dim(E)=\min\{s+t,\frac{3}{2}s+\frac{1}{2}t,s+1  \}.
\end{equation} 
For more references and history on this problem, we refer to \cite{ren2023furstenberg}, \cite{orponen2023projections}.

The Furstenberg set problem was also considered in higher dimensions; see, for example, \cite{hera2019hausdorff}, \cite{dkabrowski2022integrability}, \cite{li2024orthogonal}. It can be formulated similarly as follows.

Fix integers $1\le k<n$. Fix $s\in [0,k], t\in [0,(k+1)(n-k)]$. We say $E\subset \R^n$ is a $(s,t;n,k)$-set in $\R^n$, if $E$ is of the form
\[ E=\bigcup_{V\in \cV}Y(V), \]
where $\cV\subset A(k,\R^n)$ is a set of $k$-planes $\R^n$ with $\dim\cV=t$ and for each $V\in \cV$ there exists $Y(V)\subset V$ with $\dim Y(V)=s$. The Furstenberg set problem in $\R^n$ for $k$-planes asks us to compute
\[ \inf_{E \textup{~is~a~}(s,t;n,k)\textup{-set}}\dim(E). \]
The Furstenberg set problem in higher dimensions remains wide open.

\begin{remark}
    {\rm The original Furstenberg set problem requires the $k$-planes in $\cV$ to have distinct directions. We will not address this version, as it is more challenging than Kakeya conjecture which is open in higher dimensions. Though we still use the term ``Furstenberg set problem" to refer to this non-direction-distinct version addressed in this paper, and we hope this does not cause confusion. For progress on the direction-distinct Furstenberg set problem in finite fields, see \cite{dhar2022linear}, \cite{dhar2021simple}, \cite{dhar2024furstenberg}, \cite{ellenberg2016furstenberg}. }
\end{remark}

\bigskip

\subsubsection{Exceptional set estimate}\hfill

Next, we introduce the problem of the exceptional set estimate.
We first introduce this problem over $\R$. The projection theory dates back to Marstrand \cite{marstrand1954some}, who showed that if $A$ is a Borel set in $\R^2$, then the projection of $A$ onto almost every one-dimensional subspace has Hausdorff dimension $\min\{1,\dim A\}$. This was generalized to $\R^n$ by Mattila \cite{mattila1975hausdorff} who showed that if $A$ is a Borel set in $\R^n$, then the projection of $A$ onto almost every $k$-dimensional subspace has Hausdorff dimension $\min\{k,\dim A\}$.

One can consider a refined version of Marstrand's projection problem, which is known as the exceptional set estimate.

For $V\in G(k,\R^n)$, we use $\pi_V:\R^n\rightarrow V$ to denote the orthogonal projection. For $A\subset \R^n$ and $s>0$, we define the exceptional set
\[ E_s(A;n,k):=\{V\in G(k,\R^n): \dim(\pi_V(A))<s\}. \]
The exceptional set estimate is about finding the optimal upper bound of $\dim(E_s(A;n,k))$ (in terms of $\dim A, s,n,k$), that is, 
\[ \sup_{A\subset \R^n, \dim(A)=a}\dim(E_s(A;n,k)). \]
For technical reasons, $A$ is always assumed to be Borel.

Much progress has been made on this problem; see, for example, \cite{kaufman1968hausdorff}, \cite{falconer1982hausdorff}, \cite{peres2000smoothness}, \cite{orponen2023hausdorff}, \cite{orponen2023projections}.

In $\R^2$, D. Oberlin \cite{oberlin2012restricted} conjectured that for $0<s\le \min\{1,a\}$ and Borel set $A\subset \R^2$ with $\dim A=a$, one has the following exceptional set estimate
\begin{equation}\label{oberlin}
    \dim(E_s(A;2,1))\le \max\{2s-a,0\}. 
\end{equation} 
This conjecture was proved by Ren and Wang in \cite{ren2023furstenberg} as a corollary of \eqref{renwang}.

\bigskip

\subsection{The problems in prime fields}
In the remainder of the paper, we look only at the prime fields. Let $p$ be a prime and $\Fp$ be the prime field with $p$ elements. Throughout the paper, $A\lessapprox B$ means $A\le  C (\log p)^{C} B$ for some large constant $C$ (independent of $p$). $A\approx B$ means both $A\lessapprox B$ and $B\lessapprox A$.  $A\lesssim B$ means $A\le C B$ for a large constant $C$ (independent of $p$). $A\sim B$ means both $A\lesssim B$ and $B\lesssim A$. 

\subsubsection{Furstenberg set problem}

\begin{definition}[$(s,t;n,k)$-set]
    Fix integers $1\le k<n$, and numbers $0\le s\le k$, $0\le t\le (k+1)(n-k)$. Let $0<\lam\le 1$. We say $E$ is a $(s,t;n,k;\lam)$-Furstenberg set (over $\Fp$), if $E$ is a subset of $\Fp^n$ of form
    \[ E=\bigcup_{V\in\cV}Y(V). \]
    Here, $\cV\subset A(k,\Fp^n)$ is a set of $k$-planes with $\#\cV\ge \lam p^t$; for each $V\in\cV$, $Y(V)$ is a subset of $V$ with $\#Y(V)\ge \lam p^s$. 
\end{definition}

\begin{remark}
   {\rm We remark that $s,t,n,k$ are important parameters, while $\lam$ is not.
   Since $\lambda$ is not an important parameter, from now on we will drop $\lam$ from the notation and simply call $(s,t;n,k;\lam)$-set as $(s,t;n,k)$-set, thinking of $\lam$ as a fixed constant in $(0,1]$. 
   }
\end{remark}

We give the following definition.
\begin{definition}
    We say $(s,t;n,k)$ is admissible if $1\le k<n$ are integers and $0\le s\le k$, $0\le t\le (k+1)(n-k)$.
\end{definition}

The Furstenberg set problem can be formulated as follows:

\begin{question}
Suppose $(s,t;n,k)$ is admissible.
Find the largest $\bF(s,t;n,k)$, so that 
\begin{equation}\label{deff}
    \inf_{E:\ (s,t;n,k)\textup{-set}} \#E\gtrsim_{s,t,n,k,\e} p^{\bF(s,t;n,k)-\e} 
\end{equation}
holds for any $\e>0$.
\end{question}

In the above inequality, the implicit constant may depend on many parameters, but never depends on $p$. For simplicity, we may view $s,t,n,k$ as fixed parameters and write $\gtrsim_{s,t,n,k,\e}$ as $\gtrsim_\e$ throughout the paper. The dependence on $s,t,n,k$ will be mentioned when needed. 
We are interested in the behavior as $p$ goes to infinity, so the goal is to find the index $\bF(s,t;n,k)$ for fixed $(s,t;n,k)$.

We see that the problem consists of two parts:
for fixed $(s,t;n,k)$,
\begin{enumerate}
    \item (Lower bound) show for any $(s,t;n,k)$-set $E$, we have $\#E\gtrsim_{\e} p^{\bF(s,t;n,k)-\e}$;
    \item (Upper bound) show the existence of a $(s,t;n,k)$-set $E$, such that $\#E\lesssim p^{\bF(s,t;n,k)}$.
\end{enumerate}

\medskip

Based on Ren-Wang's result \eqref{renwang}, it is reasonable to make a conjecture in $\Fp^2$.

\begin{conjecture}\label{conj2d}
    Fix $s\in(0,1],t\in (0,2]$. Then for any $(s,t;2,1)$-set $E$ over $\Fp$, we have
    \[ \#E\gtrsim_{s,t,\e} p^{\min\{s+t,\frac{3}{2}s+\frac{1}{2}t,s+1  \}-\e}, \]
    for any $\e>0$.
\end{conjecture}

\begin{remark}
{\rm
    It is crucial to use the prime field $\Fp$. Otherwise, the conjecture fails. Consider the problem in $\F_q^2$, where $q=p^2$ is the square of a prime. We choose the set of lines $\cV=\{y=ax+b: a,b\in \Fp\}$, so $\#\cV=q$ and hence $t=1$. For each line $V: y=ax+b$ in $\cV$, define $Y(V):=\{(x,ax+b):x\in\Fp\}$, so $\#Y(V)=p=q^{\frac12}$ and therefore $s=\frac12$. We can calculate $\#\big(\bigcup_{V\in\cV}Y(V)\big)=p^2=q$. However, $\min\{s+t,\frac32s+\frac12t,s+1\}=\frac54$.

    Usually, people believe that the problem in a finite field is easier than in Euclidean space. However, Conjecture \ref{conj2d} remains open, whereas the Euclidean version has been resolved.
}
\end{remark}

\bigskip

In the paper, we explicitly find the function $\bF(s,t;n,k)$. The definition of $\bF(s,t;n,k)$ is tricky, so we postpone it later in Definition \ref{deffurindex}. The following are the main results.

\begin{theorem}[Upper bound]\label{upper bound}
    Suppose $(s,t;n,k)$ is admissible. Then for any prime $p$, there exits a $(s,t;n,k)$-set $E$ over $\Fp$, such that
    \[ \#E\lesssim p^{\bF(s,t;n,k)}. \] 
\end{theorem}

\begin{theorem}[Lower bound]\label{lower bound}
    Suppose $(s,t;n,k)$ is admissible.
    If we assume Conjecture \ref{conj2d} holds, then for any $(s,t;n,k)$-set $E$ over $\Fp$, we have
    \[ \#E\gtrsim_{\e} p^{\bF(s,t;n,k)-\e}, \]
    for any $\e>0$.
\end{theorem}

\bigskip

\subsubsection{Exceptional set estimate}\hfill

We also study the exceptional set estimate in prime fields in all dimensions. We show that the two-dimensional result for the Furstenberg set problem implies all higher-dimensional results for exceptional set estimates.

To define the projection in $\Fp^n$, we first consider an equivalent definition of the projection in $\R^n$.

Recall that $\pi_V:\R^n\rightarrow V$ is the orthogonal projection. For $V\in G(k,\R^n)$, we also define the map
\[ \pi_V^*:\R^n\rightarrow A(k,\R^n) \]
as $\pi^*_V(x):=x+V$. We see that $\pi_V^*(x)$ is the unique $k$-plane in $A(k,\R^n)$ that is parallel to $V$ and passes through $x$. It is not hard to see that $\pi_V$ and $\pi^*_{V^\perp}$ are the same things. This motivates the definition of projections in $\Fp^n$. We will still use the notations $E_s(A;n,k), \pi_V^*$ for $\Fp^n$ (which has been used for $\R^n$ in the previous subsection). This will not cause confusion.

\begin{definition}[Exceptional set]
    For $V\in G(m,\Fp^n)$, define
    \[ \pi_V^*:\Fp^n\rightarrow A(m,\Fp^n) \]
    as $\pi_V^*(x):=x+V$.

    Suppose $A\subset \Fp^n$. Let $s>0$. We define
    \[ E_s(A;n,k):=\{ V\in G(n-k,\Fp^n): \#\pi_V^*(A)< p^s \}. \]
(We remark that $E_s(A;n,k)$ also depends on $p$, but for simplicity, we drop $p$ from the notation.)
\end{definition}

\bigskip

We explicitly find the function $\bM(a,s;n,k)$, whose definition is tricky and therefore postponed to Definition \ref{defM}. We state the following results.

\begin{theorem}[Lower bound]\label{exceptionallowerthm}
    Fix $1\le k<n$, $a\in (0,n]$ and $0<s$. Then for $p$ large enough, 
there exists $A\subset \Fp^n$ with $\#A\gtrsim  p^a$, such that
\begin{equation}
    \#E_{s}(A;n,k)\gtrsim  p^{\bM(a,s;n,k)}.
\end{equation}

\end{theorem}

\begin{theorem}[Upper bound]\label{exceptionalthm}
    Assume Conjecture \ref{conj2d} holds. Fix $1\le k<n$, $a\in (0,n]$ and $0<s$. 
    If $p$ is large enough, then for any $A\subset \Fp^n$ with $\#A\gtrsim p^a$, we have
\begin{equation}
    \#E_{s-\e}(A;n,k)\lesssim_{\e} p^{\e+\bM(a,s;n,k)},
\end{equation}
for any $\e>0$.
\end{theorem}

\bigskip

\subsection{Strategy of the proof}
The proof of Theorem \ref{upper bound} and Theorem \ref{exceptionallowerthm} is by constructing examples. The proof of Theorem \ref{lower bound} and Theorem \ref{exceptionalthm} is based on a sequence of pigeonhole arguments and induction on both $n$ and $k$. We explain the idea of the proof of Theorem \ref{lower bound}.

We consider the case $(n,k)=(3,1)$. Suppose that for any $(s',t';2,1)$-set $E'$, we have $\#E'\gtrsim p^{\bF(s',t';2,1)}$. The goal is to prove for any $(s,t;3,1)$-set $E$, $\#E\gtrsim p^{\bF(s,t;3,1)}$. We will carefully analyze the structure of $E$ and find subsets of $E$ which are $(s',t';2,1)$-set. Then we use the estimate for those $(s',t';2,1)$-sets. 

Write $E=\bigcup_{L\in\cL}Y(L)$, where $\cL\subset A(1,\Fp^3)$ with $\#\cL=p^t$ and $Y(L)\subset L$ with $\#Y(L)=p^s$. Project $\cL$ to $\Fp^2$, and partition $\cL$ according to their images. For each $W\in A(1,\Fp^2)$, define $\cL_W$ to be the set of those $L\in\cL$ whose image under the projection $\Fp^3\rightarrow \Fp^2$ is $W$. We may assume that all lines in $\cL$ are not parallel to $(0,0,1)$ (by throwing away a small portion of $\cL$). We obtain the partition
\[ \cL=\bigsqcup_{W\in A(1,\Fp^2)}\cL_W. \]
By pigeonholing, we can find $\cW\subset A(1,\Fp^2)$ so that $\#\cL_W$ are comparable for all $W\in\cW$. Moreover, denoting $\#\cW=p^{t_1}$, $\#\cL_W=p^{t_2}$, we may assume $t_1+t_2=t$.

Next, we look at 
\[ E_W:=\bigcup_{L\in\cL_W}Y(L). \]
This is a $(s,t_2;2,1)$-set in $W\times \Fp\cong \Fp^2$. By hypothesis, we have
\[ \#E_W\gtrsim p^{\bF(s,t_2;2,1)}. \]

We write 
\[ E_W=\bigsqcup_{\xi\in W} E_W\cap (\xi\times \Fp). \]
By pigeonholing on $\#E_W\cap (\xi\times \Fp)$, we may assume that there is a subset $\Xi_W\subset W$ such that $\#E_W\cap (\xi\times \Fp)$ are comparable for $\xi\in\Xi_W$. Moreover, denoting $\#\Xi=p^{s_1}$, $\#E_W\cap (\xi\times \Fp)=p^{s_2}$, we may assume
\[ s_1+s_2\ge \bF(s,t_2;2,1). \]
Also, we may assume $s_1,s_2$ are the same for all $W\in\cW$ by doing another pigeonholing on $W$. We can also assume a strong condition on $s_1$: $s_1\ge s$. However, we do not expand the discussion here.

It turns out that the worst scenario for $E_W$ is when 
\[ E_W=\Xi_W\times\{1,\dots,p^{s_2}\}. \]
This will be made precise in the latter sections, so we do not expand the discussion here.

Now, we have
\[ E\supset (\bigcup_{W\in\cW}\Xi_W)\times \{1,\dots,p^{s_2}\}. \]
Note that $\bigcup_{W\in\cW}\Xi_W$ is a $(s_1,t_1;2,1)$-set. Hence,
\[ \#E\gtrsim p^{\bF(s_1,t_1;2,1)+s_2}. \]

The problem boils down to the following linear programming problem:
suppose that
\begin{equation}
    \begin{cases}
    t_1+t_2=t,\\
        s_1+s_2\ge \bF(s,t_2;2,1),\\
        s_1\ge s.
    \end{cases}
\end{equation}
Show that
\[ \bF(s_1,t_1;2,1)+s_2\ge \bF(s,t;3,1). \]

\bigskip

\subsection{Structure of the paper}
In Section \ref{sec2}, we define $\bF(s,t;n,k), \bM(a,s;n,k)$ and prove some fundamental lemmas. In Section \ref{sec3}, we prove Theorem \ref{upper bound}. In Sections \ref{sec4} and \ref{sec5}, we prove Theorem \ref{lower bound}. In Section \ref{sec6}, we prove Theorem \ref{exceptionallowerthm}. In Sections \ref{sec7} and \ref{sec8}, we prove Theorem \ref{exceptionalthm}.

\newpage

\section{Preliminary}\label{sec2}

\subsection{Definition of \texorpdfstring{$\bF(s,t;n,k)$ and $ \bM(a,s;n,k)$}{}} Now, we give the definitions of $\bF$ and $\bM$ that appeared in Theorem \ref{upper bound}, \ref{lower bound}, \ref{exceptionallowerthm}, \ref{exceptionalthm}.

\begin{definition}[Furstenberg index]\label{deffurindex}
    Suppose $(s,t;n,k)$ is admissible, we define $\bF(s,t;n,k)$ as follows.
    
    (a) If $s=0$, define
    \begin{equation}
        \bF(0,t;n,k)= \max\{0,t-k(n-k)\}.
    \end{equation}

    (b) If $s>0$, write $s=d+\si$ where $d\in\N, \si\in(0,1]$. If also $t\in [0,(k-d-1)(n-k)]$, define
    \begin{equation}
        \bF(s,t;n,k)= s.
    \end{equation}

    If we are not in Case (a) and (b), then $s\in (0,k]$ and $t\in ((k-d-1)(n-k),(k+1)(n-k)]$,
    we write 
\begin{equation}\label{expression}
    \begin{cases}
        s=d+\si\\
        t=(k-d-1)(n-k)+(d+2)m+\tau,
    \end{cases}
\end{equation} 
    where $d\in\{0,\dots,k-1\}$, $\si\in (0,1]$,  $m\in\{0,\dots,n-k-1\}$, $\tau\in (0,d+2]$. 

(c) If $\tau\in(0,2]$, define
\begin{equation}
\begin{split}
    \bF(s,t;n,k)&=s+m+\min\{\tau,\frac12(\si+\tau),1\}\\
    &= d+m+\min\{\si+\tau,\frac32\si+\frac12\tau,\si+1\}.
\end{split}
\end{equation}

(d) If $\tau\in (2,d+2]$, define
\begin{equation}
\begin{split}
    \bF(s,t;n,k)&=s+m+\min\{\tau,\frac12(\si+\tau),1\}\\
    &= s+m+1.
\end{split}
\end{equation}

\end{definition}

\begin{remark}
    {\rm In (c) and (d), $\bF$ has the same expression, but we separate them for the convenience of later discussions.
    
    Through this definition, we can calculate $\bF(s,t;2,1)=\min\{s+t,\frac32s+\frac12t,s+1\}$ for $s\in (0,1], t\in [0,2]$. We see that Conjecture \ref{conj2d} is equivalent to saying that $\inf_{E:\ (s,t;2,1)\textup{-set}} \#E\gtrsim_\e p^{\bF(s,t;2,1)-\e}$. }
\end{remark}

The following is a construction of a $(s,t;2,1)$-set, which provides evidence on why Conjecture \ref{conj2d} should be reasonable.

\begin{proposition}\label{prop2d}
For $0\le t\le 2$, there exists a $(0,t;2,1)$-set $E$ such that
    \begin{equation}
        \#E\lesssim p^{\max\{0,t-1\}}.
    \end{equation}
For $0<s\le 1, 0\le t\le 2$.\, there exists a $(s,t;2,1)$-set $E$ such that    
    \begin{equation}\label{f21}
        \#E\lesssim p^{\min\{s+t,\frac32s+\frac12t,s+1\}}.
    \end{equation}
\end{proposition}

\begin{proof}
When $s=0$, we construct $E=\bigcup_{L\in \cL} Y(L)$ as follows. If $t\le 1$, choose $\cL$ to be a set of $p^t$ lines passing through $(0,0)$ and $Y(L)=(0,0)$. Then $\#E=1=p^0$. If $t>1$, we choose $p^{t-1}$ points $E$ on $\Fp\times \{0\}$, and for each $x\in E$, we choose $p-1$ lines through $x$ that are not parallel to $\Fp\times \{0\}$. If the line $L$ is through $x\in E$, then we let $Y(L)=x$. We see that $E$ is a $(0,t;2,1)$-set with cardinality $p^{t-1}$. (Here, $p^{t-1}$ should actually be the integer $\lceil p^{t-1}\rceil$, but we still write it as $p^{t-1}$ for simplicity, thinking of it as an integer. Same convention applies in the whole paper.) 

Consider the case when $s>0$. There are three terms in the ``min" on the RHS of \eqref{f21}. We note that: when $t\le s$, minimum achieves at $s+t$; when $s\le t\le 2-s$, minimum achieves at $\frac32s+\frac12t$; when $2-s\le t$, minimum achieves at $s+1$.
\begin{enumerate}
    \item If $E=\bigcup_{L\in\cL} Y(L)$ is a $(s,t;2,1)$-set, then $\#E\le \sum_{L\in\cL} \#Y(L)\le p^{s+t}$.

    \item Choose $\cL$ to be a set of $\frac12 p^t$ lines that are not parallel to the line $\Fp\times\{0\}$. Let $Y(L)=L\cap (\Fp\times \{1,2,\dots,p^s\})$. Then $E$ is a $(s,t;2,1)$-set with $E=\bigcup_{L\in\cL} Y(L)\subset \Fp\times \{1,\dots,p^s\}$. Therefore, $\#E\le p^{s+1}$.

    \item It remains to consider the case when $s+t\le 2, s\le t$. This example is of Szemer\'edi-Trotter type. The heuristic is shown in Figure \ref{F0}, where we have $\sim p^{\frac{t-s}{2}}$ directions and for each direction, there are $\sim p^{\frac{t+s}{2}}$ lines. We also have a rectangle of size $\sim p^{s}\times p^{\frac{s+t}{2}}$. The rectangle intersects each line $L$ at $\sim p^s$ points, which is defined to be $Y(L)$. We also see $E=\bigcup_L Y(L)$ is contained in the rectangle. Therefore, $\#E\lesssim p^{\frac32s+\frac12t}$.

    Let $\cL=\{ y=ax+b: a=1, \dots,  p^{\frac{t-s}{2}},b=1, \dots,  p^{\frac{s+t}{2}} \}$. Here, we view $a,b$ as elements in $\Fp$, and hence $y=ax+b$ as a line in $\Fp^2$. For each $L\in \cL$, let $Y(L)=\{(x,y)\in L:x=1,\dots, \frac12p^s\}$. We see that $E=\bigcup_{L\in \cL}Y(L)$ is contained in $\{(x,y): x=1,\dots,  \frac12p^s, y=1,\dots, 2 p^{\frac{s+t}{2}} \}$. Therefore, $\#E\lesssim p^{\frac32s+\frac12t}$. 
\end{enumerate}
\end{proof}

\begin{figure}[ht]
\centering
\begin{minipage}[b]{0.45\linewidth}
\includegraphics[width=5cm]{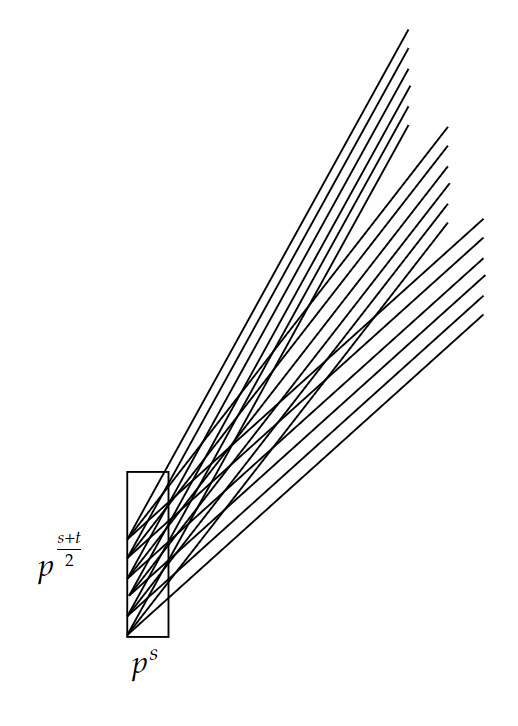}
\caption{}
\label{F0}
\end{minipage}
\end{figure}

\bigskip

Next, we define the index for Marstrand's orthogonal projection.

\begin{definition}[Marstrand index]\label{defM}
Let $1\le k<n$ be integers. Let $a\in (0,n], s>0$.
We give the following definition for $\bM(a,s;n,k)$. We write $a=m+\beta, s=l+\ga$ where $m,l\in\N$ and $\beta,\ga\in(0,1]$. For convenience, we call such an expression of $a,s$ canonical.

\textbf{\textup{Type 1}}: If $s>\min\{a,k\}$, define
\begin{align}
    \bM(a,s;n,k)=k(n-k).
\end{align}

\textbf{\textup{Type 2}}: If $s\le\min\{a,k\}$, $l+1\le m\le n+l-k$ and $ \ga\in(\beta,1]$, define 
\begin{align}
    \bM(a,s;n,k)=k(n-k)-(m-l)(k-l)+\max\{2\ga-(\beta+1),0\}.
\end{align}

\textbf{\textup{Type 3}}: If $s\le\min\{a,k\}$, $l\le m\le n+l-k-1$ and $ \ga\in(0,\beta] $, define
\begin{align}
\bM(a,s;n,k)=k(n-k)-(m+1-l)(k-l)+\max\{2\ga-\beta,0\}.
\end{align}

\textbf{\textup{Type 4}}: If $s\le a-(n-k)$, define
\begin{align}
    \bM(a,s;n,k)=-\infty.
\end{align}
\end{definition}

\begin{remark}
{\rm $\bM(a,s;n,k)=-\infty$ corresponds to the case when the exceptional set is empty. We remark that in the Euclidean setting, it is assumed that the empty set has dimension $-\infty$, to satisfy a Fubini-type heuristic: $\dim(A\times B)=\dim A+ \dim B$.
}    
\end{remark}

One can check, when $(n,k)=(2,1)$,
\begin{equation}
    \bM(a,s;2,1)=\begin{cases}
        1 & s>\min\{1,a\}\\
        \max\{0,2s-a\} & a-1<s\le \min\{1,a\}\\
        -\infty & s\le a-1.
    \end{cases}
\end{equation}

We discuss a construction of an exceptional set.

\begin{figure}[ht]
\centering
\begin{minipage}[b]{0.45\linewidth}
\includegraphics[width=5cm]{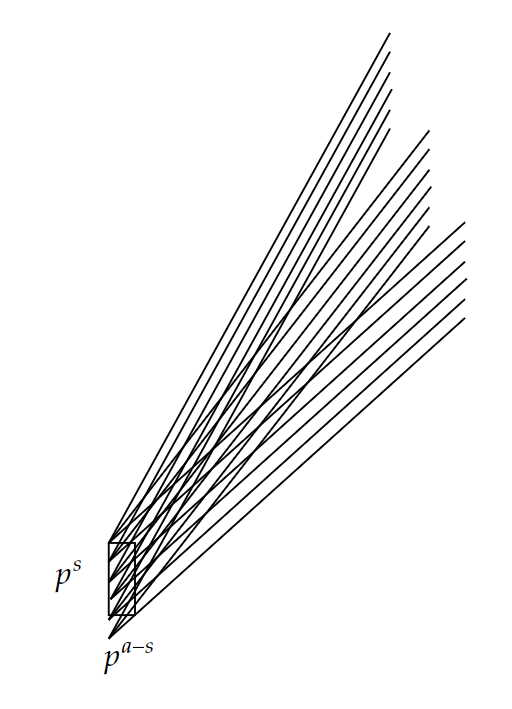}
\caption{}
\label{M0}
\end{minipage}
\end{figure}

\begin{proposition}\label{oberlin2d}
    Let $a\in (0,2]$, $s\in (\frac{a}{2},\min\{1,a\}]$. Then for $p$ large enough, there exists $A\subset \Fp^2$ with $\#A\gtrsim p^a$, such that
    \[ \#E_s(A;2,1)\gtrsim p^{2s-a}. \]
\end{proposition}
\begin{proof}
The heuristic is shown in Figure \ref{M0}. Our set $A$ is the lattice points in the rectangle of size $\frac15 p^{a-s}\times \frac15 p^s$. For any slope $k\in [1,p^{2s-a}]$, we can find $<p^s$ lines with slope $k$ to cover $A$. Hence, this slope-$k$ direction belongs to the exceptional set.  

Consider the set $A=\{(x,y): |x|\le \frac15p^{a-s}, |y|\le \frac15p^s \}\subset \Fp^2$. For $k,m\in\Fp$, we use $\ell_{k,m}$ to denote the line $ y=kx+m$. We claim that for each $|k|\le \frac15p^{2s-a}$,  
\[\pi_{\ell_{k,0}}^*(A)\subset \{\ell_{k,m}:|m|< \frac12 p^s\}.\]
Note that if $(x,y)\in A$ and $|k|\le \frac15p^{2s-a}$, then $|y-kx|\le \frac12 p^s$, which means that there exists $|m|< \frac12 p^s$ such that $(x,y)\in \ell_{k,m}$.
Hence, $\#\pi_{\ell_{k,0}}^*(A)<p^s$. This means $E_s(A;2,1)\supset \{\ell_{k,0}: |k|\le \frac15 p^{2s-a}\}$.

We see that $\#A\gtrsim p^a$, and $E_s(A;2,1)$
has cardinality $\gtrsim p^{2s-a}$.
\end{proof}

\subsection{Useful lemmas} We prove several useful lemmas, which will be used later in the proof. We suggest that the reader skip this part in the first reading.

\begin{lemma} Suppose $G(k,\Fp^n)$ is the set of $k$-dimensional subspaces in $\Fp^n$, $A(k,\Fp^n)$ is the set of $k$-planes in $\Fp^n$. We have: 
\begin{enumerate}
    \item $\#G(k,\Fp^n)\sim p^{k(n-k)}$;
    \item $\# A(k,\Fp^n)\sim p^{(k+1)(n-k)}$;
    \item Fix integers $l,k$ with $l+k\ge n$. Let $W_0\in A(l,\Fp^n)$. Then $\#\{ V\in A(k,\Fp^n): V\cap W_0 \textup{~is~a~}(l+k-n)\textup{-plane} \}\gtrsim p^{(k+1)(n-k)}$.
\end{enumerate}
    
\end{lemma}
\begin{proof}
    (1) and (2) are fundamental.
    We prove (3). $V\cap W_0$ is a $(l+k-n)$-plane means $V$ intersects $W_0$ transversely. (3) agrees with intuition in $\R^n$, since generic $V\in A(k,\R^n)$ intersect $W_0$ transversely. Noting (2), we see that it suffices to prove 
    \begin{equation}\label{rplane}
        \#\{ V\in A(k,\Fp^n): V\cap W_0 \textup{~is~a~}r\textup{-plane} \}\lesssim p^{(k+1)(n-k)-1}, 
    \end{equation} 
    for $r>l+k-n$.
    The proof of \eqref{rplane} is as follows. Choose a $r$-plane $L$ in $W_0$. There are $\#G(r,\Fp^l)\sim p^{(r+1)(l-r)}$ choices for $L$. If $V\cap W_0=L$, then $V\supset L$. We have $\le \#\{V\in A(k,\Fp^n): V\supset L\}=\#G(k-r,\Fp^{n-r})\sim p^{(k-r)(n-k)}$ choices for $V$ such that $V\cap W_0=L$. The two estimates give \eqref{rplane}.
\end{proof}

\begin{lemma}\label{easylem2}
Let $0\le l\le n$, $1\le k,m\le n$ be integers that satisfy $n-k\ge m-l$ and $l\le k,m$. If $W\in G(m,\Fp^n)$, then
\[ \#\{ V\in G(n-k,\Fp^n):\#(\pi^*_V(W))\le p^l \}\sim p^{k(n-k)-(k-l)(m-l)}. \]
\end{lemma}
\begin{proof}
By definition
\begin{align*}
    \{ V\in G(n-k,\Fp^n):\#(\pi_V^*(W))\le p^l \}&=\{V\in G(n-k,\Fp^n):\#(V\cap W)\ge \#W/p^l= p^{m-l}\},
\end{align*} 
The latter set is
\[ \bigcup_{j\ge m-l}\bigcup_{U\in G(j,W)}\{V\in G(n-k,\Fp^n):V\supset U\}. \]

It suffices to show for $j\ge m-l$,
\begin{equation}\label{grassineq}
    \#\bigcup_{U\in G(j,W)}\{V\in G(n-k,\Fp^n):V\supset U\}\sim p^{k(n-k)-(k+j-m)j}.
\end{equation}

On the one hand, the LHS is 
\begin{equation}
    \begin{split}
        &\le \sum_{U\in G(j,W)}\#\{V\in G(n-k,n):V\supset U\}= \#G(j,\Fp^m)\#G(n-k-j,\Fp^{n-j})\\
        &\sim p^{j(m-j)}p^{(n-k-j)k}=p^{k(n-k)-(k+j-m)j}.
    \end{split}
\end{equation} 

On the other hand, for fixed $U\in G(j,W)$, we have
\begin{equation}
    \begin{split}
        &\{V\in G(n-k,\Fp^n): V\cap W=U\}\\
        \supset &\{V\in G(n-k,\Fp^n): V\supset U\}\setminus \bigcup_{j+1\le j'\le n-k}\bigcup_{U'\in G(j',W), U'\supset U}\{V\in G(n-k,\Fp^n): V\supset U'\} 
    \end{split}
\end{equation} 
One can check $\#\{V\in G(n-k,\Fp^n): V\supset U\}=\#G(n-k-j,\Fp^{n-j})=p^{k(n-k-j)}$. $\bigcup_{j+1\le j'\le n-k}\bigcup_{U'\in G(j',W), U'\supset U}\{V\in G(n-k,\Fp^n): V\supset U'\} $ has cardinality
\begin{equation}
    \begin{split}
        &\le \sum_{j+1\le j'\le n-k}\#G(j'-j,\Fp^{m-j})\#G(n-k-j',n-j')\\
        &\sim \sum_{j+1\le j'\le n-k}p^{(j'-j)(m-j')+(n-k-j')k}\\
        &=\sum_{j+1\le j'\le n-k}p^{k(n-k-j)+(j'-j)(m-k-j')}\\
        &\lesssim p^{k(n-k-j)-1}.
    \end{split}
\end{equation} 
In the last line, we use $(j'-j)(m-k-j')\le m-k-j-1\le m-k-(m-l)-1\le -1$.
Therefore, 
\[ \#\{V\in G(n-k,\Fp^n): V\cap W=U\}\sim p^{k(n-k-j)}. \]

We can bound the left-hand side of \eqref{grassineq} by
\[ \ge \#\bigsqcup_{U\in G(j,W)}\{V\in G(n-k,\Fp^n): V\cap W=U\}\sim p^{j(m-j)+k(n-k-j)}=p^{k(n-k)-(k+j-m)j}. \]
\end{proof}

\begin{lemma}
    $\bF(s,t;n,k)$ is increasing in $s$ and $t$.
\end{lemma}

\begin{lemma}\label{easybound}
    Suppose $(s,t;n,k)$ is admissible, then 
    \[ \bF(s,t;n,k)\ge s+\max\{0,t-k(n-k)\}. \]
\end{lemma}
\begin{proof}
It is easy to check $\bF(s,t;n,k)\ge s$, so we just need to check $\bF(s,t;n,k)\ge s+t-k(n-k)$. We check four cases as in Definition \ref{deffurindex}.

    (a) If $s=0$, this is true by definition.

    (b) If $s>0$, write $s=d+\si.$ If $t\in [0,(k-d-1)(n-k)]$, then
    \[ \bF(s,t;n,k)=s\ge s+t-k(n-k). \]

    (c) If \eqref{expression} holds and $\tau\in (0,2]$, then
    \[ \bF(s,t;n,k)=s+m+\min\{\tau,\frac12(\si+\tau),1\}. \]
    To show it is $\ge s+t-k(n-k)$ is equivalent to show
    \[ (d+1)(n-k-m)+\min\{\tau,\frac12\si+\frac12\tau,1\}\ge \tau. \]
Since $(d+1)(n-k-m)\ge 1\ge \frac12\tau$ and $\min\{\tau,\frac12\si+\frac12\tau,1\}\ge \frac12\tau$, it is proved.

    (d) In the final case, we have
    \[ \bF(s,t;n,k)=s+m+1. \]
    To show it is $\ge s+t-k(n-k)$ is equivalent to show
    \[ (d+1)(n-k-m)+1\ge \tau. \]
    This is true since $(d+1)(n-k-m)\ge d+1$ and $\tau\le d+2$.
\end{proof}

\begin{lemma}\label{lipint}
$\bF(s,t;n,k)$ is $1$-Lipschitz in $t$. More precisely,
suppose that $(s,t_1+t_2;n,k)$ is admissible, $t_1,t_2\ge 0$. Then
    \[ \bF(s,t_1+t_2;n,k)\le t_1+\bF(s,t_2;n,k). \]
\end{lemma}
\begin{proof}
    This is equivalent to showing that for fixed $s,n,k$, the map $t\mapsto \bF(s,t;n,k)$ is Lipschitz with Lipschitz constant $\le 1$. We sketch the proof. First, it is not hard to show that $t\mapsto \bF(s,t;n,k)$ is continuous. Since $\bF(s,t;n,k)$ is piecewise defined, we check each formula of $\bF(s,t;n,k)$. We see that the derivative on $t$ is $\le 1$.   
\end{proof}

$\bF(s,t;n,k)$ is continuous in $t$, but not in $s$. However, we can still say something about $s$.
\begin{lemma}\label{lipins}
    $\bF(s,t;n,k)$ is locally left-Lipschitz in $s$. In other words, if we write $s=d+\si$ ($d\in\N, \si\in(0,1]$), we have
    \[ \bF(s-\theta,t;n,k)\ge \bF(s,t;n,k)-O(\theta) \]
    for any $0\le \theta< \si$.
\end{lemma}
\begin{proof}
    We just note that $s\mapsto\bF(s,t;n,k)$ is defined piecewise for $s$ in left-open intervals of the form $(d,d+1]$ where $d\in\N$. Moreover, $s\mapsto\bF(s,t;n,k)$ is Lipschitz in each of these left-open intervals.
\end{proof}

The next results are for the Marstrand index.

\begin{lemma}
Each pair $(a,s)$ satisfying $a\in (0,n], s>0$ belongs to exactly one of the four types in Definition \ref{defM}.
\end{lemma}
\begin{proof}
     To show this, we just need to prove that for a fixed $a\in (0,n]$, the ranges of $s$ determined by four cases form a partition of $(0,\infty)$. Type 1 corresponds to $s\in (\min\{a,k\},\infty)$. Type 4 corresponds to $s\in (0,a-(n-k)]$. Note that $a-(n-k)\le \min\{a,k\}$, so we just need to show the ranges of $s$ determined by Type 2 and Type 3 from a partition of $(a-(n-k),\min\{a,k\}]$. Since the ranges of $s$ determined by Type2 and Type 3 are disjoint, we just need to show
    \[ \{s: s\textup{~satisfies~Type~}  2\}\bigcup \{s: s\textup{~satisfies~Type~} 3\}=(a-(n-k),\min\{a,k\}]. \]
    
    We first show ``$\subset"$. Suppose $s$ satisfies the condition in Type 2. We want to show $s>a-(n-k)$, which is equivalent to $m<n+l-k+\ga-\beta$. This is true since $m\le n+l-k$ and $\ga>\beta$. Suppose $s$ satisfies the condition in Type 3. Again, we want to show $m<n+l-k+\ga-\beta$. This is true since $m\le n+l-k-1$ and $\ga-\beta> -\beta\ge -1$.

    Next, we show ``$\supset$". Suppose $s=l+\ga\in (a-(n-k),\min\{a,k\}]$. $s>a-(n-k)$ implies
    \[ m<n+l-k+\ga-\beta. \]
    $s\le a$ implies
    \[ m\ge l+\ga-\beta. \]
    If $\ga\in (\beta,1]$, then $l+1\le m\le n+l-k$. Hence, $s$ satisfies the condition in Type 2. If $\ga\in (0,\beta]$, then $l\le m\le n+l-k-1$. Hence, $s$ satisfies the condition in Type 3.
    \end{proof}

\begin{lemma}
    $\bM(a,s;n,k)$ is decreasing in $a$ and increasing in $s$.
\end{lemma}

\begin{lemma}\label{lipina}
    For $(a,s,n,k)$ and $(a-\theta,s-\theta;n,k)$ that are in the domain of $\bM$ and $\theta\ge 0$, we have
    \[ \bM(a-\theta,s-\theta;n,k)\le \bM(a,s;n,k). \]
\end{lemma}
\begin{proof}
    Let $a=m+\beta,s=l+\ga$ be the canonical expression. We just check four types as in Definition \ref{defM}.

Type 1: If $s>\min\{a,k\}$, then 
\[ \bM(a-\theta,s-\theta;n,k)\le k\le \bM(a,s;n,k). \]

Type 4: If $s\le a-(n-k)$, then we have $\bM(a-\theta,s-\theta;n,k)=\bM(a,s;n,k)=-\infty$.

Type 2: Suppose $(a,s;n,k)$ is of Type 2.

If $\theta<\beta$.
\begin{align*}
    \bM(a-\theta,s-\theta;n,k)=k(n-k)-(m-l)(k-l)+\max\{2(\ga-\theta)-(\beta-\theta+1),0\}\\
    \le k(n-k)-(m-l)(k-l)+\max\{2\ga-(\beta+1),0\}=\bM(a,s;n,k). 
\end{align*}

If $\beta\le \theta<\ga$, then $a-\theta=(m-1)+(1+\beta-\theta), s-\theta=l+(\ga-\theta)$. So, $(a-\theta,s-\theta;n,k)$ is of Type 3. We have the same estimate:
\begin{align*}
    \bM(a-\theta,s-\theta;n,k)=k(n-k)-(m-l)(k-l)+\max\{2(\ga-\theta)-(\beta-\theta+1),0\}\\
    \le k(n-k)-(m-l)(k-l)+\max\{2\ga-(\beta+1),0\}=\bM(a,s;n,k). 
\end{align*}
Note that the integer part $(m,l)$ is reduced to $(n-1,l)$, so we can continue doing the iteration.

Type 3: This is similar to Type 2, so we just omit the proof.
    
\end{proof}

\begin{lemma}\label{easyM}
    Suppose $a-(n-k)<s\le \min\{k,a\}$. Write $a=m+\beta,s=l+\ga$ as in Definition \ref{defM}. Then
    \[ \bM(a,s;n,k)\le k(n-k)-(m-l)(k-l)+\max\{2\ga-(\beta+1),0\}, \]
    and
    \[ \bM(a,s;n,k)\ge k(n-k)-(m+1-l)(k-l)+\max\{2\ga-\beta,0\}. \]
\end{lemma}
\begin{proof}
    The proof is simply by checking when $(a,s;n,k)$ is of Type 2 or Type 3.
\end{proof}

We also show two easy results for the exceptional set estimate.

\begin{lemma}
    Recall Definition \ref{defM}. Suppose that we are in Type 1: $s>\min\{k,a\}$. Then there exists $A\subset \Fp^n$ with $\#A\sim p^a$, such that
    \[ \#E_s(A;n,k)\gtrsim p^{k(n-k)}. \]
\end{lemma}
\begin{proof}
    We just choose $A$ to be any set with cardinality $p^a$. One can check by definition that $E_s(A;n,k)=G(k,\Fp^n)$.
\end{proof}

\begin{lemma}
    Recall Definition \ref{defM}. Suppose that we are in Type 4: $s\le a-(n-k)$. Then for any $A\subset \Fp^n$ with $\#A\ge p^a$, we have
    \[ \#E_s(A;n,k)=0. \]
\end{lemma}
\begin{proof}
   If $V\in E_s(A;n,k)$, then 
   \[ p^a=\#A=\sum_{L\in \pi_V^*(A)}\#(L\cap A)\le\#\pi_V^*(A)\cdot \#V<p^{s+n-k}, \]
   a contradiction! Therefore, $E_s(A;n,k)=\emptyset$.
\end{proof}

\

\section{Furstenberg set problem: upper bound}\label{sec3}

In this section, we prove Theorem \ref{upper bound}. In other words, we will construct a $(s,t;n,k)$-set $E=\bigcup_{V\in\cV}Y(V)$, such that
\[ \#E\lesssim_{s,t,n,k} p^{\bF(s,t;n,k)}.\]

\subsection{Two special cases} To familiarize ourselves with $\bF(s,t;n,k)$, we discuss two special cases: $s=k$; $k=1$.

When $s=k$, we can morally view $Y(V)$ as $V$. Therefore, a $(k,t;n,k)$-set is of the form
\[ E=\bigcup_{V\in\cV}V. \]
Estimating the size of $E$ is known as the problem of ``union of $k$ planes". It has been resolved in finite fields by R. Oberlin \cite{oberlin2016unions}. The Euclidean version was resolved by the author in \cite{gan2023hausdorff}.

When $k=1$, this is the Furstenberg set problem for lines. To construct a $(s,t;n,1)$-set 
\[ E=\bigcup_{L\in\cL}Y(L) \]
with size as small as possible, we hope that the lines in $\cL$ are compressed in a low-dimensional space. 
Write $t=2m+\beta$, where $0\le m\le n-2$ is an integer, and $\beta\in (0,2]$. Since $t\le 2(m+1)$, we can make $\cL$ contained in $\Fp^{m+2}$. Therefore, our goal becomes to construct a small $(s,t;m+2,1)$-set. The trick is to take the Cartesian product of a $(s,\beta;2,1)$-set with $\Fp^m$. Let $E'\subset \Fp^2$ be a $(s,\beta;2,1)$-set and let $E'$ have the form
\[ E'=\bigcup_{L'\in \cL'}Y(L'). \]
We will construct a $(s,t;m+2,1)$-set $E$ based on $E'$. Let $\pi_{\Fp^2}:\Fp^{m+2}\rightarrow \Fp^2$ be the projection onto the first two coordinates. Define $\cL$ to be the set of lines $L$ in $\Fp^{m+2}$ such that $\pi_{\Fp^2}(L)$ is a line in $\cL'$. Then, $\#\cL\sim \# G(1,\Fp^{m+1})\cdot \#\cL'\sim p^{2m+\beta}=p^t$. If $L\in \cL$, then $L'=\pi_{\Fp^2}(L)\in \cL'$. We define $Y(L):=L\cap \pi_{\Fp^2}^{-1}(Y(L'))$. We see that $\#Y(L)=\#Y(L')\sim p^s$, so 
\[ E=\bigcup_{L\in\cL}Y(L) \]
is a $(s,t;m+2,1)$-set.
Also, since $E\subset E'\times \Fp^m$, we have
\[ \#E\le\#E'\cdot p^m\lesssim p^{m+\min\{s+\beta,\frac32s+\frac12\beta,s+1\}}  \]
By checking Definition \ref{deffurindex}, one exactly sees that
    \[ m+\min\{s+\beta,\frac32s+\frac12\beta,s+1\}=\bF(s,t;m+2,1)=\bF(s,t;n,1). \]

\subsection{Proof of Theorem \ref{upper bound}}\hfill

For admissible $(s,t;n,k)$, we will construct a $(s,t;n,k)$-set 
\[ E=\bigcup_{V\in\cV}Y(V) \]
such that
\[ \#E\gtrsim p^{\bF(s,t;n,k)}. \]

\fbox{Case (a): \texorpdfstring{$s=0$}{}}

If $t\le k(n-k)$, we choose $\cV$ to be $p^t$ $k$-planes pass through the origin $\bzz^n$, and $Y(V)=\{\bzz^n\}$. Then $\#E\le 1=p^0$. Here, $\bzz^n$ means $(\underbrace{0,\dots,0}_{n \textup{~times~}})$.

If $t\ge k(n-k)$, we choose $p^{t-k(n-k)}$ points $X$ in $\Fp^{n-k}$. For each such point $x\in X$, we choose $\sim p^{k(n-k)}$ $k$-planes that pass through $x$, and require the $k$-planes to be transverse to $\Fp^{n-k}$. We obtain $\sim p^t$ $k$-planes which is $\cV$. If $V$ passes through $x\in X$, we set $Y(V)=\{x\}$. We see that $\#E=\#X\le p^{t-k(n-k)}$.

\bigskip

\fbox{Case (b): $s>0, t\in [0,(k-d-1)(n-k)]$}

We write $s=d+\si$, where $d\in \{0,\dots,k-1\}, \si\in(0,1]$.
The construction is similar to Case (a). We choose $\sim p^t$ many $k$-planes that contain $\Fp^{d+1}$. This is doable since 
\[ \#\{ V\in A(k,\Fp^n): V\supset L \}=\#G(k-d-1,\Fp^{n-d-1})\sim p^{(k-d-1)(n-k)}. \]
Let $\cV$ be these $k$-planes. Fix $Y\subset \Fp^{d+1}$ with $\#Y=p^s$. For each $V\in \cV$, let $Y(V)=Y$. We see that $\#E=\#Y=p^s$.

\bigskip

\fbox{Case (c): $s\in (0,k], t\in ((k-d-1)(n-k),(k+1)(n-k)]$}

We write 
\begin{equation}\label{stformula}
    \begin{cases}
        s=d+\si\\
        t=(k-d-1)(n-k)+(d+2)m+\tau,
    \end{cases}
\end{equation} 
    where $d\in\{0,\dots,k-1\}$, $\si\in (0,1]$,  $m\in\{0,\dots,n-k-1\}$, $\tau\in [0,2]$. We remark that here we allow $\tau=0$, although we assumed $\tau>$ in the canonical expression. Allowing $\tau=0$ will produce the example in Case (d).

We write $\Fp^n=\Fp^{n-k+d+1}\times \F_p^{k-d-1}$. Note that $\gtrsim 1$ fraction of $k$-planes intersect $\Fp^{n-k+d+1}$ at a $(d+1)$-plane. Heuristically speaking, to make $\#E$ as small as possible, we would like $E$ to be a subset of $\F_p^{n-k+d+1}\times \bzz^{k-d-1}$. For simplicity, we write $\F_p^{n-k+d+1}\times \bzz^{k-d-1}$ as $\Fp^{n-k+d+1}$, thinking of it as a subspace of $\Fp^n$. This will not cause any ambiguity.

Note that for each $W\in A(d+1,\F_p^{n-k+d+1})$, the cardinality of 
\[\{V\in A(k,\Fp^n): V\supset W,  V \textup{~is~transverse~to~} \F_p^{n-k+d+1} \}\] 
is $\sim \#G(k-d-1,\Fp^{n-k})\sim p^{(k-d-1)(n-k)}$. We just need to find $\cW\subset A(d+1,\F_p^{n-k+d+1})$ with $\#\cW\sim p^{t-(k-d-1)(n-k)}=p^{(d+2)m+\tau}$, and for each $W\in\cW$, to find $Y(W)\subset W$ with $\#Y(W)\sim p^{s}$, such that
\begin{equation}\label{constructW}
    \#\bigcup_{W\in\cW} Y(W)\lesssim p^{\bF(s,t;n,k)}.
\end{equation}

If this is done, then we let $\cV$ be those $k$-planes $V$ that contain some $W\in\cW$, and $Y(V)=Y(W)$ if $V\supset W$. We see that $\#\cV\sim \#\{W\}\cdot p^{(k-d-1)(n-k)}\sim p^t$, and hence $E=\bigcup_{V\in\cV} Y(V)$ is a $(s,t;n,k)$-set. Also, we have the upper bound
\[ \#E=\#\bigcup_{W\in\cW} Y(W)\lesssim p^{\bF(s,t;n,k)}. \]

It suffices to find a $(s,(d+2)m+\tau;n-k+d+1,d+1)$-set $E_1=\bigcup_{W\in\cW} Y(W)$ such that \eqref{constructW} holds. 
Note that $\#A(d+1,\Fp^{m+d+2})\sim p^{(d+2)(m+1)}$ and $(d+2)m+\tau\le (d+2)(m+1)$. We will construct the $(d+1)$-planes $\cW$ so that they lie in $\Fp^{m+d+2}$. 
It suffices to find a $(s,(d+2)m+\tau;m+d+2,d+1)$-set $E_1=\bigcup_{W\in\cW} Y(W)$ such that \eqref{constructW} holds.

We write $\Fp^{m+d+2}=\Fp^{d+2}\times\Fp^{m}=\Fp^2\times \Fp^d\times \Fp^m$. We first choose a $(\si,\tau;2,1)$-set $E_2=\bigcup_{L\in\cL}Y(L)$ in $\Fp^2$, such that
\[\#E_2\lesssim p^{\min\{\si+\tau,\frac32\si+\frac12\tau,\si+1\}}. \]
The existence of $E_2$ was shown in Proposition \ref{prop2d}.

Next, we define $\cU:=\{L\times \Fp^d: L\in\cL\}$ which is a subset of $A(d+1,\Fp^{d+2})$. And for each $U=L\times \Fp^d\in\cU$, define $Y(U)=Y(L)\times \Fp^d$. We see that $\#\cU= \#\cL\sim p^\tau$, $\#Y(U)=\#Y(L)\cdot p^d\sim p^{\si+d}=p^s$. We also see that $E_3:=\bigcup_{U\in \cU}Y(U)$ is contained in $E_2\times \Fp^d$, and hence \[\#E_3\lesssim p^{d+\min\{\si+\tau,\frac32\si+\frac12\tau,\si+1\}}.\]

Finally, we define $\cW$. For each $U\in \cU$, consider the $(m+d+1)$-plane $U\times \Fp^m$.
Let $\pi_U:U\times \Fp^m\rightarrow U$ be the projection (naturally defined). We let $\cW_U$ be the set of $(d+1)$-planes $W$ contained in $U\times \Fp^m$ such that $\pi_U(W)=U$. Note that $\gtrsim 1$ fraction of the $(d+1)$-planes in  $U\times \Fp^m$ lie in $\cW_U$, so $\#\cW_U\sim \#A(d+1,\Fp^{m+d+1})\sim p^{(d+2)m}$.

We define $\cW=\bigcup_{U\in \cU}\cW_U$ (actually this is a disjoint union). For $W\in \cW_U$, define $Y(W)=W\cap \pi_U^{-1}(Y(U))$. We see that $\#\cW\sim \#\cU\cdot p^{(d+2)m}\sim p^{(d+2)m+\tau}$, and $\#Y(W)=\#Y(U)\sim p^s$. Hence, $E_1=\bigcup_{W\in\cW}Y(W)$ is a $(s,(d+2)m+\tau;m+d+2,d+1)$-set. Since $E_1\subset E_3\times \Fp^m$, we have
\[ \#E_1\le \#E_3\cdot p^m\lesssim p^{d+m+\min\{\si+\tau,\frac32\si+\frac12\tau,\si+1\}}. \]
The construction is done.

\bigskip
\fbox{Case (d)}

We still use \eqref{stformula}, but the range of $\tau$ is $\tau\in(2,d+2]$. Note that the worst case is when $\tau=d+2$.
What we need to do is for
\begin{equation}
    \begin{cases}
        s=d+\si\\
        t=(k-d-1)(n-k)+(d+2)(m+1),
    \end{cases}
\end{equation} 
construct a $(s,t;n,k)$-set $E$ such that
\[ \#E\lesssim p^{s+m+1}. \]
The trick is to still use the construction in Case (c), with $\tau$ replaced by $0$ and $m$ replaced by $m+1$. Therefore, we can construct a $(s,t;n,k)$-set $E$ such that
\[\#E\lesssim p^{d+m+1+\min\{\si+0,\frac32\si+\frac120,\si+1\}=p^{d+m+1+\si}=p^{s+m+1}}. \]
\

\section{Furstenberg set problem: Lower bound I}\label{sec4}

To become more familiar with the problem, we first prove two special cases of Theorem \ref{lower bound}: $s=0$; $t=0$.

\begin{proposition}\label{s0}
    If $E$ is a $(0,t;n,k)$-set, then
    \[ \# E\gtrsim p^{\bF(0,t;n,k)}=p^{\max\{0,t-k(n-k)\}}. \]
\end{proposition}

\begin{proof}
    Suppose $E=\bigcup_{V\in\cV}Y(V)$. Note that each $x\in E$ can lie in at most $\#G(k,\Fp^n)\sim p^{k(n-k)}$ different $V$. We have
    \[ \#E\cdot p^{k(n-k)}\ge \sum_{V\in\cV}\#Y(V)\gtrsim p^t. \]
    Hence, $\#E\gtrsim p^{t-k(n-k)}$. Also, trivially, $\#E\ge 1$.
\end{proof}

\begin{proposition}\label{t0}
    If $E$ is a $(s,0;n,k)$-set, then
    \[ \# E\gtrsim p^{\bF(s,0;n,k)}=p^{s}. \]
\end{proposition}
\begin{proof}
    By the definition of $(s,0;n,k)$-set,
    $\#E\ge \#Y(V)\gtrsim p^s$.
\end{proof}

In this section, we prove Theorem \ref{lower bound} for $(n,k)=(k+1,k)$. The proof is by induction on $k$. The goal of this section is to prove the following result.

\begin{proposition}\label{proplower1}
    Let $k\in \N$ and $k\ge 2$.
    Suppose for admissible $(s,t;k,k-1)$,  we know
    \begin{equation}\label{conditionprop}
        \inf_{E:\ (s,t;k,k-1)\textup{-set}}\#E\gtrsim_\e p^{\bF(s,t;k,k-1)-\e}
    \end{equation} 
    for any $\e>0$.
    Then for admissible $(s,t;k+1,k)$, we have
    \begin{equation}
        \inf_{E:\ (s,t;k+1,k)\textup{-set}}\#E\gtrsim_\e p^{\bF(s,t;k+1,k)-\e},
    \end{equation} 
    for any $\e>0$.
\end{proposition}

We mention a technical thing in the proof.
For $\eta>0$, denote $\eta\Z:=\{\eta n: n\in\Z\}$.
 In Theorem \ref{lower bound}, the implicit constant in the inequality also depends on $s,t$, and therefore may not be uniformly bounded. To get rid of it, we choose a small $\eta>0$, and discretize $s,t$, so that $s,t\in \eta\Z$. Then we obtain the bound
\[ \#E\gtrsim_{\e,\eta,n,k}p^{\bF(s-\eta,t-\eta;n,k)-\e}, \]
where the implicit constant is independent of $s,t$.

\begin{lemma}\label{lempig1}
Assume \eqref{conditionprop} holds.
Suppose $(s,t;k+1,k)$ is admissible. Fix $0<\e<1/100$.
    Then for any $(s,t;k+1,k)$-set $E$, there exist numbers $t_1,t_2,s_1,s_2,u,v$, such that
    \begin{equation}
        \begin{cases}
            1\le p^{t_1}\le p^{k-1}, 1\le p^{t_2}\le p^2;\\
            1\le p^{s_1}\le p, 1\le p^{s_2}\le p^s;\\
            p^{s_1}\lesssim p^u\le p, 1\le p^v\le p;
        \end{cases}
    \end{equation}
    and
    \begin{equation}
        \begin{cases}
            p^{t_1+t_2}\gtrapprox p^{t};\\
            p^{s_1+s_2}\gtrapprox p^{s};\\
            p^{u+v}\gtrsim_\e p^{\bF((s_1-\e^2)_+,(t_2-\e^2)_+;2,1)-\e};
        \end{cases}
    \end{equation}
    and
    \begin{equation}\label{lempig1ineq}
        \#E \gtrsim_\e p^{u+\max\{ \bF((s_2-\e^2)_+,(t_1+v-\e^2)_+;k,k-1),s_2+v \}-\e}.
    \end{equation}
Here, $(x)_+=\max\{0,x\}$. In addition, the implicit constants do not depend on $t_1,t_2,s_1,s_2,u,v$.
\end{lemma}

\begin{proof}
    We will do several steps of pigeonholings, which lead to the parameters $t_1,t_2,s_1,s_2,u,v$.

    By the definition of $(s,t;k+1,k)$-set, we write
    \[ E=\bigcup_{V\in\cV}Y(V), \]
    where $\cV\subset A(k,\Fp^{k+1})$ with $\#\cV\sim p^t$, and $Y(V)\subset V$ with $\#Y(V)\sim p^s$.

    The strategy is to find a subset of $\cV$, and a subset of $Y(V)$ for each $V$, still denoted by $\cV$, $Y(V)$, so that the new $\cV, Y(V)$ satisfy certain uniformity. This is done by a sequence of pigeonhole arguments.

    By properly choosing the coordinate, we can assume $\gtrsim 1$ fraction of the $k$-planes in $\cV$ are not parallel to $\Fp^k$. Throwing away some $V\in\cV$, we can just assume all the $k$-planes in $\cV$ are not parallel to $\Fp^k$.

    Now, we partition $\cV$ according to $G(k-1,\Fp^k\times \bzz)$. Since each $V\in \cV$ is not parallel to $\Fp^k$ and the ambient space is $\Fp^{k+1}$, $V$ must intersect $\Fp^k\times \bzz$ at a $(k-1)$-plane. For each $W\in G(k-1,\Fp^k\times \bzz)$, we define $\cV_W$ to be the set of $V\in\cV$ such that $V\cap \Fp^k$ is parallel to $W$. We obtain a partition
    \[ \cV=\bigsqcup_{W\in G(k-1,\Fp^k\times \bzz)}\cV_W. \]

    \textit{First pigeonholing.} We assume that there exists $\cW\subset G(k-1,\Fp^k\times \bzz)$ with $\#\cW\sim p^{t_1}$ such that for each $W\in\cW$, $\#\cV_W\sim p^{t_2}$. Here, $1\le p^{t_1}\le p^{k-1},1\le p^{t_2}\le p^{2}$ are dyadic numbers, and 
    \begin{equation}
        p^{t_1}p^{t_2}\gtrapprox p^t.
    \end{equation}
     (Recall that $A\gtrapprox B$ means $A\gtrsim (\log p)^C B$.) We still use $\cV$ to denote $\bigcup_{W\in\cW} \cV_W$. And we see that
     \[ \#\cV\gtrapprox p^t. \]

    \textit{Second pigeonholing.} By properly choosing the coordinate and throwing away $\lesssim 1$ fraction of $W$ in $\cW$, we can assume that each $W$ in $\cW$ satisfies $W\cap (\Fp\times \bzz^k)=\bzz^{k+1}$. Denote $\wt{\Fp^2}=\Fp\times \bzz^{k-1}\times \Fp$. If $V\in\cV_W$ and $W\in \cW$, then $V\cap \wt{\Fp^2}$ is a line. We denote it by $\ell_V:=V\cap \wt{\Fp^2}$. Note that $\ell_V$ is not parallel to the first coordinate in $\wt{\Fp^2}$.
    
    For a fixed $V\in\cV_W$, we do a Fubini-type argument. Write 
    \[V=\bigsqcup_{x\in \ell_V} (x+W).\] 
    Noting that $\#Y(V)\sim p^s$, by pigeonholing, we find a subset $Y(\ell_V)\subset \ell_V$, so that $\#Y(\ell_V)\sim s_1(V)$, and for each $x\in Y(\ell_V)$, $\#Y(V)\cap (x+W)\sim s_2(V)$. Here, $1\le p^{s_1(V)}\le p, 1\le p^{s_2(V)}\le p^s$ are dyadic numbers that satisfy
    \[ p^{s_1(V)}\cdot p^{s_2(V)}\gtrapprox p^s. \]
    By throwing away some horizontal slices of $Y(V)$, we can assume
    \[ Y(V)=\bigsqcup_{x\in Y(\ell_V)}\bigg(Y(V)\cap(x+W)\bigg). \]
    For the new $Y(V)$, we still have
    \[ \#Y(V)\gtrapprox p^s. \]
For convenience, given $V\in\cV_W$ and $x\in\ell_V$, we denote
\[ Y(V,x):=Y(V)\cap (x+W). \]
$Y(V,x)$ is called the horizontal $x$-slice. (Note that there is no confusion since $V$ uniquely determines $W$.)  
Thus, we have
\begin{equation}
    Y(V)=\bigsqcup_{x\in Y(\ell_V)}Y(V,x).
\end{equation}
    
    \textit{Third pigeonholing.} For each $W\in \cW$, we can find dyadic numbers $p^{s_1(W)}, p^{s_2(W)}$ so that $\gtrapprox 1$ fraction of $V$ in $\cV(W)$ satisfy $p^{s_1(V)}=p^{s_1(W)}, p^{s_2(V)}=p^{s_2(W)}$.

    \textit{Fourth pigeonholing.} There exist dyadic numbers $p^{s_1},p^{s_2}$ such that $\gtrapprox 1$ fraction of $W$ in $\cW$ satisfy $p^{s_1(W)}=p^{s_1}$ and $p^{s_2(W)}=p^{s_2}$.

    \medskip 

    We summarize what we have done so far.
    We can redefine $\cW$, $\cV_W$ and $\cV$, so that $\cV=\bigcup_{W\in\cW}\cV_W$ and the following uniformity hold. 
    \[ \#\cW\gtrapprox p^{t_1}. \]
    \[ \#\cV_W\gtrapprox p^{t_2}, \textup{~for~}W\in\cW. \]
    For $W\in\cW$ and each $V\in\cV_W$,
    \begin{equation}\label{YV}
        Y(V)=\bigsqcup_{x\in Y(\ell_V)}Y(V,x), 
    \end{equation} 
    with $\#Y(\ell_V)\sim p^{s_1}$ and for each $x\in Y(\ell_V)$, $\# Y(V,x)\sim p^{s_2}$.

    Fix a $W\in\cW$. We want to estimate $\bigcup_{V\in\cV_W} Y(\ell_V)(\subset \wt{\Fp^2})$. This is some sort of Furstenberg set in two-dimensional space. Also, an important property is that each $\ell_V$ is not parallel to the first coordinate. We will exploit some ``quasi-product" structure of this set. 

    \begin{lemma}\label{lemuv}
        Assume Conjecture \ref{conj2d} holds.
        Suppose $E=\bigcup_{\ell\in\cL} Y(\ell)$ is a $(s_1,t_2;2,1)$-set. In addition, we assume that each $\ell\in\cL$ is not parallel to $\Fp\times \bzz$. Then there exist dyadic numbers $p^u,p^v$ such that
        \[ p^{s_1}\lesssim p^u\lesssim p,  \]
        \[ 1\le p^v\le p, \]
         and
        \[ p^up^v\gtrsim_\e p^{\bF(s_1,t_2;2,1)-\e}. \]
        for any $\e>0$.
        In addition, $E$ contains a subset of the form
        \[ \bigsqcup_{\xi\in\Xi}Y_\xi, \]
        where $\Xi\subset \bzz\times\Fp$ with $\#\Xi\sim p^u$, and $Y_\xi\subset E\cap\bigg(\xi+(\Fp\times \bzz)\bigg)$ with $\#Y_\xi\sim p^v$.
    \end{lemma}

    \begin{proof}
        For $\xi\in \bzz\times \Fp$, define $E(\xi):=E\cap\bigg(\xi+(\Fp\times \bzz)\bigg)$.
        By pigeonholing on horizontal slices of $E$, we can find $\Xi\subset \bzz\times \Fp$ so that for all $\xi\in\Xi$, $E(\xi)$ have comparable cardinality. We may denote $\#\Xi\sim p^u, \#E(\xi)\sim p^v$ for two dyadic numbers $p^u,p^v\ge 1$. We can also assume $p^u p^v\gtrapprox \#E$. Hence, assuming Conjecture \ref{conj2d}, we have
        \[ p^u p^v \gtrsim_\e p^{\bF(s_1,t_2;2,1)-\e}. \]
        However, we still need the condition $p^u\gtrsim p^{s_1}$. This is handled by the following trick. We will perform an algorithm to find a sequence $E=E_1\supset E_2\supset\cdots\supset E_m$, and disjoint subsets $\Xi_1,\Xi_2\cdots,\Xi_m$ of $\bzz\times \Fp$.
        To start with, let $E_1=E$. By pigeonholing, we can find $\Xi_1\subset \bzz\times \Fp$, such that $\#\Xi_1\sim p^{u_1}$, $\#E(\xi)\sim p^{v_1}$ for each $\xi\in\Xi$, and
        \[ p^{u_1}p^{v_1}\gtrapprox \#E_1\gtrsim_\e p^{\bF(s_1,t_2;2,1)-\e}. \]
        Suppose that we have defined $E_{i}, \Xi_{i}$ for $1\le i\le j-1$. If $\sum_{i=1}^{j-1}p^{u_i}\ll p^{s_1}$, then we define
        \[ E_j:=E\setminus \bigcup_{1\le i\le j-1}\bigsqcup_{\xi\in\Xi_{i}}E(\xi). \]
        Since each $\ell\in\cL$ is not parallel to $\Fp\times\bzz$, we see that $E_j$ is still a $(s_1,t_2;2,1)$-set. By the same pigeonholing, we find $\Xi_j\subset (\bzz\times \Fp)\setminus \bigcup_{1\le i\le j-1}\Xi_i$, so that $\#\Xi_j\sim p^{u_j}$, $\#E(\xi)\sim p^{v_j}$ for each $\xi\in\Xi_j$, and
        \[  p^{u_j}p^{v_j}\gtrsim_\e p^{\bF(s_1,t_2;2,1)-\e}. \]
        If $\sum_{i=1}^{j-1}p^{u_i}\gtrsim p^{s_1}$, then we stop and let $m=j-1$.

        The algorithm will stop in finite steps. Denote $\Xi:=\bigcup_{1\le i\le m}\Xi_i$. Let $p^u,p^v$ be dyadic numbers so that $p^v=\min_{1\le i\le m}p^{v_i}$, $p^u\sim\#\Xi$. For each $\xi\in\Xi$, let $Y_\xi$ be a subset of $E(\xi)$ with $\#Y_\xi\sim p^v$. This is doable, since $\#E(\xi)\gtrsim p^v$ if $\xi\in\Xi$. One can check
        \[ p^u p^v\gtrsim_\e p^{\bF(s_1,t_2;2,1)-\e}. \]
        This is proved in the following way. Suppose $p^v=p^{v_i}$, then
        \[ p^up^v=p^up^{v_i}\ge p^{u_i}p^{v_i}\gtrsim_\e p^{\bF(s_1,t_1;2,1)-\e}. \]
        Also, one can check
        \[ p^u=\sum_{i=1}^mp^{u_i}\gtrsim p^{s_1}. \]

    \end{proof}

    Return to $\wt E_W:=\bigcup_{V\in\cV_W} Y(\ell_V)$ (which lies in $\wt{\Fp^2}=\Fp\times \bzz^{k-1}\times \Fp$). Recall $\#\cV_W\gtrapprox p^{t_2}$ and $\#Y(\ell_V)\sim p^{s_1} $.
    We see that $\wt E_W$ is a $(s_1,(t_2-\e)_+;2,1)$-set, for any $\e>0$.

    By Lemma \ref{lemuv}, there exist dyadic numbers $p^{u(W)}, p^{v(W)}$, so that $\wt E_W$ contains a subset of form
    \[ \bigsqcup_{\xi\in \Xi_W} Y_W(\rho_\xi). \]
Here, $\Xi_W\subset \bzz^k\times \Fp$ with $\#\Xi_W\sim p^{u(W)}$; $\rho_\xi:=\xi+(\Fp\times \bzz^k)$ and $Y_W(\rho_\xi)\subset \rho_\xi$ with $\#Y_W(\rho_\xi)\sim p^{v(W)}$. We also have
\[ p^{u(W)}\gtrsim p^{s_1}. \]
\[ p^{u(W)}p^{v(W)}\gtrsim_\e p^{\bF(s_1,(t_2-\e)_+;2,1)-\e}. \]
By pigeonholing on $(u(W),v(W))$ again, we can throw away some $W\in\cW$ so that the cardinality of $\cW$ is still $\gtrapprox p^{t_1}$, and assume that there exist $u,v$ such that
\[ (u(W),v(W))=(u,v),\ \textup{~for~all~}W\in\cW, \]
and
\[ p^u\gtrsim p^{s_1}. \]
\[ p^up^v\gtrsim_\e p^{\bF(s_1,(t_2-\e)_+;2,1)-\e}. \]
By Lemma \ref{lipint},
\[ p^up^v\gtrsim_\e p^{\bF(s_1,t_2;2,1)-\e}. \]

Finally, we are about to estimate the lower bound of the cardinality of $E=\bigcup_{W\in\cW}\bigcup_{V\in\cV_W}Y(V).$ This $E$ is not the original one, but the one after pigeonholing. Recall \eqref{YV}, we have
\begin{equation}
    \begin{split}
        \bigcup_{V\in\cV_W} Y(V)&=\bigcup_{V\in\cV_W} \bigsqcup_{x\in Y(\ell_V)}Y(V,x)\\
        &=\bigcup_{x\in \wt E_W}\bigg(\bigcup_{V: V\in\cV_W, x\in Y(\ell_V)}Y(V,x)\bigg).
    \end{split}
\end{equation} 
We define
\[ Y(W,x):=\bigcup_{V: V\in\cV_W, x\in Y(\ell_V)}Y(V,x).\]
Then $Y(W,x)$ is contained in the $(k-1)$-plane $x+W$. 

If $x\in \wt E_W$, then
\[ \# Y(W,x)\gtrsim p^{s_2}. \]
Recall that 
\[ \wt E_W \supset \bigsqcup_{\xi\in \Xi_W} Y_W(\rho_\xi). \]
We have
\[ \bigcup_{V\in\cV_W} Y(V)\supset \bigcup_{\xi\in\Xi_W}\bigcup_{x\in Y_W(\rho_\xi)}Y(W,x). \]
The reader can check that, for fixed $\xi\in\Xi_W$, $\bigcup_{x\in Y_W(\rho_\xi)}Y(W,x)$ lies in $\Fp^k\times \{\xi\}$.

Now we see that
\[ E\supset \bigcup_{W\in\cW}\bigcup_{\xi\in\Xi_W}\bigcup_{x\in Y_W(\rho_\xi)}Y(W,x). \]
We want to estimate the lower bound of the right-hand side. Intuitively, one can imagine that the worst case is when $\Xi_W$ are morally the same for all $W\in\cW$; otherwise we may translate them to the same $(k+1)$-th coordinate so that their union becomes smaller. For example, if we have sets $A_i\times \{\xi_i\}\subset \Fp^k\times \Fp$ ($i\in I$). Then $\bigcup_{i\in I}A_i\times \{\xi_i\}$ is the smallest when all the $\xi_i$ are the same.

Let us return to the estimate of 
\[E':=\bigcup_{W\in\cW}\bigcup_{\xi\in\Xi_W}\bigcup_{x\in Y_W(\rho_\xi)}Y(W,x).\]
There are two lower bounds. The first lower bound is obtained from a fixed $W$. We have
\[ \#E'\ge \bigcup_{\xi\in\Xi_W}\bigcup_{x\in Y_W(\rho_\xi)}Y(W,x)\gtrsim p^up^vp^{s_2}. \]

For the second lower bound, we will estimate $E'$ on each $\xi$-slice.
We define the $\xi$-slice of $E'$ by
\[ E'_\xi:=E'\cap (\Fp^k\times \{\xi\})=\bigcup_{W: W\in\cW,\xi\in\Xi_W}\bigcup_{x\in Y_W(\rho_\xi)}Y(W,x). \]
We will see that $E'_\xi$ is some sort of $(s_2,t_\xi;k,k-1)$-set. Define 
\[ p^{t_\xi}:=\#\{ (W,x): W\in\cW, \xi\in\Xi_W, x\in Y_W(\rho_\xi) \}. \]

We note two bounds for $t_\xi$:
\begin{equation}
    p^{t_\xi}\le \#\cW\cdot \#Y_W(\rho_\xi)\lesssim p^{t_1}p^v;
\end{equation}
\begin{equation}
    \sum_{\xi\in\bzz^k\times \Fp} p^{t_\xi}=\sum_{W\in\cW}\#\Xi_W\cdot\#Y_W(\rho_\xi)\gtrapprox p^{t_1}p^u p^v.
\end{equation}
We used that $p^{t_1}\gtrsim\#\cW\gtrapprox p^{t_1}$, $\#\Xi_W\sim p^u$, $\#Y_W(\rho_\xi)\sim p^v$.

And note that $Y(W,x)$ is a subset of the line $x+W$ such that $\#Y(W,x)\gtrsim p^{s_2}$. Therefore, $E_\xi'$ is a $(s_2,t_\xi;k,k-1)$-set. By \eqref{conditionprop}, we have
\[ \#E'_\xi\gtrsim_\e p^{\bF(s_2,t_\xi;k,k-1)-\e}. \]
Therefore,
\[ \#E\gtrsim_\e \sum_{\xi\in \bzz^k\times \Fp}p^{\bF(s_2,t_\xi;k,k-1)-\e}. \]

We claim that the right-hand side is
\[ \gtrapprox p^up^{\bF(s_2,t_1+v;k,k-1)-\e}. \]
By pigeonholing, we can find a dyadic number $p^{t_\circ}\lesssim p^{t_1}p^v$ and $\Xi_\circ\subset \bzz^k\times \Fp$, so that $2 p^{t_\circ}\ge p^{t_\xi}\ge p^{t_\circ}$ for $\xi\in\Xi_\circ$ and
\[ \#\Xi_\circ\cdot p^{t_\circ}\gtrapprox p^{t_1}p^up^v. \]
Hence,
\[ \#\Xi_\circ\gtrapprox p^u p^{t_1+v-t_\circ}. \]
We can estimate
\[ \#E\gtrsim_\e \sum_{\xi\in\Xi_\circ}p^{\bF(s_2,t_\xi;k,k-1)-\e}\ge\#\Xi_\circ\cdot p^{\bF(s_2,t_\circ;k,k-1)-\e}\gtrapprox p^u p^{\bF(s_2,t_\circ;k,k-1)+t_1+v-t_\circ-\e}. \]
By Lemma \ref{lipint}, we have
\[ \#E\gtrapprox_\e p^{u+\bF(s_2,t_1+v;k,k-1)-\e}. \]

Note that the implicit constants in the above arguments may depend on various parameters $s_1,s_2,t_1,t_2$ and so on. To get rid of this dependence, we use the trick discussed before this lemma. Hence, we update the estimate to
\[ p^{u+v}\gtrsim_\e p^{\bF((s_1-\e^2)_+,(t_2-\e^2)_+;2,1)-\e}, \]
\[ \#E \gtrsim_\e p^{u+\max\{ \bF((s_2-\e^2)_+,(t_1+v-\e^2)_+;k,k-1),s_2+v \}-\e}. \]
Now, the implicit constants only depend on $s,t,k,\e$.
\end{proof}

The next lemma is a recursive inequality of $\bF(s,t;k+1,k)$, which can be used to bound the right-hand side of \eqref{lempig1ineq}.
\begin{lemma}\label{recursionF}
    Suppose $(s,t;k+1,k)$ is admissible.
    Suppose there exist numbers $t_1,t_2,s_1,s_2,u,v$, such that
    \begin{equation}
        \begin{cases}
            t_1\in [0,k-1], t_2\in [0,2];\\
            s_1\in [0,1], s_2\in [0,s];\\
            u\in [s_1,1], v\in [0,1];
        \end{cases}
    \end{equation}
    and
    \begin{equation}\label{bigger}
        \begin{cases}
            t_1+t_2\ge t;\\
            s_1+s_2\ge s;\\
            u+v\ge\bF(s_1,t_2;2,1);
        \end{cases}
    \end{equation}

    Then, 
    \begin{equation}
        u+\max\{\bF(s_2,t_1+v;k,k-1),s_2+v\}\ge \bF(s,t;k+1,k).
    \end{equation}

\end{lemma}

Let us quickly see how Lemma \ref{lempig1} and Lemma \ref{recursionF} combined to prove Proposition \ref{proplower1}.
\begin{proof}[Proof of Proposition \ref{proplower1} assuming Lemma \ref{lempig1} and Lemma \ref{recursionF}]\hfill

If $s=0$ or $t=0$, by Proposition \ref{s0} or \ref{t0}, we are done. So, we suppose $s,t>0$.

Fix $0<\e<\min\{\si/2,t/2, 1/2\}$, where $s=d+\si$ with $d\in\N, \si\in(0,1]$. We just need to prove
\begin{equation}
        \#E\gtrsim_\e p^{\bF(s,t;k+1,k)-C\e}
    \end{equation} 
for any $(s,t;k+1,k)$-set $E$. Here, $C$ is a large number.    

We only need to prove for large enough $p$, since for small $p$ we can choose the constant to be large. 

Suppose \eqref{conditionprop} holds. Then in particular, Conjecture \ref{conj2d} also holds. If $E$ is a $(s,t;k+1,k)$-set, by Lemma \ref{lempig1}, there exist numbers $t_1,t_2,s_1,s_2,u,v$, such that
    \begin{equation}
        \begin{cases}
            t_1\in [0,k-1], t_2\in [0,2];\\
            s_1\in [0,1], s_2\in [0,s];\\
            u\in [s_1-\e,1], v\in [0,1];
        \end{cases}
    \end{equation}
    and
    \begin{equation}
        \begin{cases}
            t_1+t_2\ge t-\e;\\
            s_1+s_2\ge s-\e;\\
            u+v\ge \bF(s_1,t_2;2,1)-\e;
        \end{cases}
    \end{equation}
    and
    \begin{equation}
        \#E \gtrsim_\e p^{u+\max\{ \bF(s_2,t_1+v;k,k-1),s_2+v \}-\e}.
    \end{equation}
Here, we renamed $t_1$ to be $(t_1-\e^2)_+$ in Lemma \ref{lempig1} (and similar for $t_2,s_1,s_2,u,v$). Also, we used the fact that when $p$ is large enough (depending on $\e$), then $p^{t_1+t_2}\gtrapprox p^t$ implies $t_1+t_2\ge t-\e$. Also, by our assumption. $t-\e,s-\e>0$. Let $t'=t-\e$. $s'=s-\e$.

Since $\bF(s_1,t_2;2,1)\le 2$, we can find $u\le u'\le \min\{u+\e, 1\}$ and $v\le v'\le \min\{v+\e,1\}$, so that $u'\ge s_1$,
$u'+v'\ge \bF(s_1,t_2;2,1)$. Now we have
\begin{equation}
        \begin{cases}
            t_1\in [0,k-1], t_2\in [0,2];\\
            s_1\in [0,1], s_2\in [0,s];\\
            u'\in [s_1,1], v'\in [0,1];
        \end{cases}
    \end{equation}
    and
    \begin{equation}
        \begin{cases}
            t_1+t_2\ge t';\\
            s_1+s_2\ge s';\\
            u'+v'\ge \bF(s_1,t_2;2,1);
        \end{cases}
    \end{equation}
    and
    \begin{equation}
        \#E \gtrsim p^{u+\max\{ \bF(s_2,t_1+v;k,k-1),s_2+v \}-\e}\ge p^{u'+\max\{ \bF(s_2,t_1+v';k,k-1),s_2+v' \}-O(\e)}.
    \end{equation}
In the last inequality, we used Lemma \ref{lipint}.

Finally, by Lemma \ref{recursionF}, we have
\[ u'+\max\{ \bF(s_2,t_1+v';k,k-1),s_2+v'\}\ge \bF(s',t';k+1,k)\ge \bF(s,t;n,k)-O(\e). \]
The last inequality is by Proposition \ref{lipins} and \ref{lipint}.

In summary, we finish the proof of
\[ \#E\gtrsim p^{\bF(s,t;k+1,k)-C\e}. \]
    
\end{proof}

It remains to prove Lemma \ref{recursionF}.
It can be reduced to the following lemma, with ``$\ge$" in \eqref{bigger} all replaced by ``$=$".

\begin{lemma}\label{recursionF1}
    Suppose $(s,t;k+1,k)$ is admissible.
    Suppose there exist numbers $t_1,t_2,s_1,s_2,u,v$, such that
    \begin{equation}
        \begin{cases}
            t_1\in [0,k-1], t_2\in [0,2];\\
            s_1\in [0,1], s_2\in [0,s];\\
            u\in [s_1,1], v\in [0,1];
        \end{cases}
    \end{equation}
    and
    \begin{equation}
        \begin{cases}
            t_1+t_2= t;\\
            s_1+s_2= s;\\
            u+v=\bF(s_1,t_2;2,1);
        \end{cases}
    \end{equation}

    Then, 
    \begin{equation}
        u+\max\{\bF(s_2,t_1+v;k,k-1),s_2+v\}\ge \bF(s,t;k+1,k).
    \end{equation}
\end{lemma}

We show that Lemma \ref{recursionF1} implies Lemma \ref{recursionF}.
\begin{proof}[Proof of Lemma \ref{recursionF} assuming Lemma \ref{recursionF1}]
    The idea is to slightly modify $s_1,s_2,t_1,t_2,u,v$ so that the ``$\ge$" in \eqref{bigger} becomes ``$=$", and then we can apply Lemma \ref{recursionF1}.

    For $i=1,2$, we can find $ t_i'\in [0, t_i], s'_i\in [0,s_i]$, so that $t_1'+t_2'=t, s_1'+s_2'=s$. 
    It is slightly trickier to modify $u,v$. Noting that $u+v\ge \bF(s_1,t_2;2,1)\ge \bF(s_1',t_2';2,1)\ge s_1'$, we can find $u'\in [s_1',u], v'\in[0,v]$ so that $u'+v'=\bF(s_1',t_2';2,1)$. Now we can apply Lemma \ref{recursionF1} to the numbers $t_1',t_2',s_1',s_2',u',v'$. We obtain
    \[ u'+\max\{\bF(s'_2,t'_1+v';k,k-1),s'_2+v'\}\ge \bF(s,t;k+1,k). \]
    Since the new numbers are smaller than the old numbers and by using the monotonicity of $\bF(\cdot,\cdot;n,k)$, we have
    \[u+\max\{\bF(s_2,t_1+v;k,k-1),s_2+v\}\ge \bF(s,t;k+1,k).\]
\end{proof}

It remains to prove Lemma \ref{recursionF1}.

\begin{proof}[Proof of Lemma \ref{recursionF1}]
For convenience, we recall the definition of Furstenberg index in the special case $(n,k)=(k+1,k)$.

(a) If $s=0$, 
    \begin{equation*}
        \bF(0,t;k+1,k)= \max\{0,t-k\}.
    \end{equation*}

    (b) If $s>0$, write $s=d+\si$. If also $t\in [0,k-d-1]$,
    \begin{equation*}
        \bF(s,t;k+1,k)=s.
    \end{equation*}

    If we are not in Case (a) and (b), then $s\in (0,k]$ and $t\in (k-d-1,k+1]$,
    we write 
\begin{equation*}
    \begin{cases}
        s=d+\si\\
        t=k-d-1+\tau,
    \end{cases}
\end{equation*} 
    where $d\in\{0,\dots,k-1\}$, $\si\in (0,1]$, $\tau\in (0,d+2]$. 

(c) If $\tau\in(0,2]$, 
\begin{equation*}
    \bF(s,t;k+1,k)= d+\min\{\si+\tau,\frac32\si+\frac12\tau,\si+1\}
\end{equation*}

(d) If $\tau\in (2,d+2]$, 
\begin{equation*}
    \bF(s,t;k+1,k)= s+1.
\end{equation*}

\bigskip

Let us denote 
\[ \bF:=u+\max\{\bF(s_2,t_1+v;k,k-1),s_2+v\}. \]
Our goal is to prove
\[ \bF\ge \bF(s,t;k+1,k). \]

    When $s=0$, then $s_1=s_2=0$. We want to prove $\bF\ge \max\{0,t-k\}$. Trivially, $\bF\ge 0$. Also,
    \[ \bF\ge u+\bF(0,t_1+v;k,k-1)\ge u+t_1+v-(k-1). \]
    Using $u+v= \bF(s_1,t_2;2,1)\ge t_2-1$ and $t_1+t_2=t$, we obtain
    \[ \bF\ge t-k. \]
    
    When $t\in[0, k-d-1]$, we have
    \[ \bF\ge u+s_2+v= s_2+\bF(s_1,t_2;2,1)\ge s_2+s_1= s=\bF(s,t;k+1,k). \]

    So we consider $s>0$ and $t>k-d-1$. Write
    \begin{equation*}
    \begin{cases}
        s=d+\si\\
        t=k-d-1+\tau,
    \end{cases}
\end{equation*} 
    where $d\in\{0,\dots,k-1\}$, $\si\in (0,1]$, $\tau\in (0,d+2]$. In this case,
    \[  \bF(s,t;k+1,k)=d+\min\{\si+\tau,\frac32\si+\frac12\tau,\si+1\}. \]
    We just need to show either $\bF\ge d+\si+\tau$, or $\bF\ge d+
 \frac32\si+\frac12\tau$, or $\bF\ge d+\si+1$.
    
\medskip

We first analyze $\bF(s_2,t_1+v;k,k-1)$. This is done by discussing the values of $(s_2,t_1+v)$.

\medskip

    \fbox{Case (a): $s_2=0$} By definition, we have $\bF(s_2,t_1+v;k,k-1)= \max\{0,t_1+v-(k-1)\}$. Since $s_1+s_2= s=d+\si$ and $s_1\le 1$, we have $d=0,s_1=\si$. We have
    \[ \bF \ge u+\max\{t_1+v-(k-1),v\}. \]

Recall $u+v\ge \bF(\si,t_2;2,1), t_1+t_2= t=k-1+\tau$. We have
\[ \bF\ge \bF(\si,t_2;2,1)+ 
\max\{\tau-t_2,0\}.\]
If $t_2\ge \tau$, then $\bF(\si,t_2;2,1)\ge \bF(\si,\tau;2,1)$. If $t_2\le \tau$, then $\bF(\si,t_2;2,1)+ 
\tau-t_2\ge \bF(\si,\tau;2,1)$, by Lemma \ref{lipint}.
Therefore, we have
\[  \bF\ge \bF(\si,\tau;2,1)=\min\{\si+\tau,\frac32\si+\frac12\tau,\si+1\}=\bF(s,t;k+1,k).\]
The last inequality is because $d=0$.

\medskip
    
    \fbox{ Case (b) $t_1+v\le k-d-2$ }
    We have
    \[ \bF\ge u+v+s_2\ge \bF(s_1,t_2;2,1)+s_2. \]

    \fbox{Subcase (b.0): $s_1=0$} In this case, $ \bF(s_1,t_2;2,1)=\max\{0,t_2-1\} $, $s_2=d+\si$. We have
    \[  \bF\ge \max\{0,t_2-1\}+d+\si\ge d+\si+\tau. \]
    The last inequality is because 
    \[ t_2= t-t_1\ge (k-d-1+\tau)-(k-d-2-v)\ge 1+\tau. \]

    \fbox{Subcase (b.1): $\bF(s_1,t_2;2,1)= s_1+t_2$} We have
    \[ \bF\ge s_1+t_2+s_2 = s+t_2\ge d+\si+\tau. \]
    since $t_2\ge 1+\tau$.
    
    \fbox{Subcase (b.2): $\bF(s_1,t_2;2,1)=\frac32s_1+\frac12t_2$} We have
    \[ \bF\ge \frac12(s_1+t_2)+d+\si \ge d+\frac32\si+\frac12\tau.\]
    Since $t_2\ge 1+\tau\ge \si+\tau.$

    \fbox{Subcase (b.3): $\bF(s_1,t_2;2,1)= s_1+1$} We have
    \[ \bF\ge s+1=d+\si+1. \]

\medskip

    \fbox{ Case (c): $s_2>0$, $k-d-2<t_1+v\le k-d$} We write $t_1+v=k-d-2+\tau'$ for $\tau'\in (0,2]$. Since $t_1+t_2= k-d-1+\tau$, we have $\tau'=1+\tau+v-t_2$.

    \fbox{Subcase (c.I): $s_1<\si$ } We have $s_2= s-s_1=d+(\si-s_1)$. We have
    \begin{equation}
        \begin{split}
            \bF(s_2,t_1+v;k,k-1)=\bF(d+(\si-s_1),t_1+v;k,k-1)\\
            = d+\min\{ \si-s_1+\tau',\frac32(\si-s_1)+\frac12\tau',(\si-s_1)+1 \}. 
        \end{split}
    \end{equation} 

    \fbox{Subsubcase (c.I.1): $\bF(s_2,t_1+v;k,k-1)=s_2+\tau'$} We have
    \begin{equation}
        \begin{split}
            \bF&\ge u+\max\{ s_2+\tau',s_2+v \}\\
            &=u+s_2+\max\{ \tau',v \}.
        \end{split}
    \end{equation} 

    \fbox{Subsubsubcase (c.I.1.0): $s_1=0$} Then $u+v= \bF(0,t_2;2,1)=\max\{0,t_2-1\}$. We have
    \[ \bF\ge u+s_2+\tau'=u+v+s_2+1+\tau-t_2\ge s_2+\tau=s+\tau. \]

    \fbox{Subsubsubcase (c.I.1.1): $u+v= \bF(s_1,t_2;2,1)=s_1+t_2$} We have
    \[ \bF\ge u+s_2+\tau'=u+s_2+1+\tau+v-t_2= s+\tau+1. \]

    \fbox{Subsubsubcase (c.I.1.2): $u+v=\bF(s_1,t_2;2,1)= \frac32s_1+\frac12t_2$} We have
    \[ \bF\ge \frac32s_1+\frac12t_2+s_2+1+\tau-t_2=s+\frac12(s_1+t_2)+1+\tau-t_2\ge s+1+\tau-\frac12t_2.  \]
    We also have
    \[ \bF\ge u+v+s_2\ge s+\frac12t_2. \]
    Combining together gives
    \[ \bF\ge s+\frac12(1+\tau)\ge d+\frac32\si+\frac12\tau. \]

    \fbox{Subsubsubcase (c.I.1.3): $u+v=\bF(s_1,t_2;2,1)= s_1+1$} We have
    \[ \bF\ge u+v+s_2\ge s+1. \]

\medskip

    \fbox{SubSubcase (c.I.2): $\bF(s_2,t_1+v;k,k-1)=d+\frac32(\si-s_1)+\frac12\tau'=s_2+\frac12(\si-s_1+\tau')$} We have
    \[ \bF\ge u+s_2+\max\{ \frac12(\si-s_1+\tau'),v \}=u+s_2+\max\{\frac12(\si+1+\tau+v-t_2-s_1),v\}. \]
    
\fbox{Subsubsubcase (c,I.2.0): $s_1=0$} Then $u+v=\max\{0,t_2-1\}$. We have
\[ \bF\ge \frac12u+s_2+\frac12(u+v+\si+1+\tau-t_2-s_1)\ge s_2+\frac12(\si+\tau-s_1)=d+\frac32\si+\frac12\tau.  \]

\fbox{Subsubsubcase (c.I.2.1): $u+v= s_1+t_2$} We have
\[ \bF\ge \frac12u+
 s_2+\frac12(\si+\tau+1)\ge d+\frac32\si+\frac12\tau. \]
Here, we used $u\ge s_1, 1\ge s_1$.

\fbox{Subsubsubcase (c.I.2.2): $u+v= \frac32s_1+\frac12t_2$} We have
\[ \bF\ge \frac12u+s_2+\frac12(\si+1+\tau)+\frac14(s_1-t_2) \]
To prove $\bF\ge d+\frac32\si+\frac12\tau$,
it suffices to prove 
\[ \frac12(u+1)+\frac14(s_1-t_2)\ge s_1. \]
This is true since $u+1\ge u+v= \frac32s_1+\frac12t_2$.

\fbox{Subsubsubcase (c.I.2.3): $u+v= s_1+1$ }
\[ \bF\ge u+v+s_2= s+1. \]

\fbox{Subsubcase (c.I.3): $\bF(s_2,t_1+v;k,k-1)=s_2+1$ } We have
\[ \bF\ge u+s_2+1\ge s_1+s_2+1=s+1. \]

\medskip

\fbox{Subcase (c.II): $s_1\ge \si$ } We write $s_2=(d-1)+(\si+1-s_1)$. We have
\begin{align*}
    \bF(s_2,t_1+v;k,k-1)&=d-1+\min\{ \si+1-s_1+\tau', \frac32(\si-s_1+1)+\frac12\tau', \si-s_1+2 \}\\
    &=d+\min\{ \si-s_1+\tau', \frac32(\si-s_1)+\frac12\tau'+\frac12, \si-s_1+1 \}\\
    &\ge d+\min\{ \si-s_1+\tau', \frac32(\si-s_1)+\frac12\tau', \si-s_1+1 \}
\end{align*} 
The discussion is the same as Subcase (c.I), since we did not use the condition $s_1<\si$.

\fbox{Subcase (d): $t_1+v>k-d$} We have
\[ \bF(s_2,t_1+v;k,k-1)=s_2+1. \]
Therefore,
\[ \bF\ge u+s_2+1\ge s_1+s_2+1=s+1. \]

\end{proof}

\

\section{Furstenberg set problem: Lower bound II}\label{sec5}

\begin{proposition}\label{proplower2}
    Let $n,k\in\N$ and $n\ge k+2$.
    Suppose for any $n'=k+1,\dots,n-1$ and admissible $(s,t;n',k)$, we know
    \begin{equation}\label{conditionprop2}
        \inf_{E:\ (s,t;n',k)\textup{-set}}\#E\gtrsim_\e p^{\bF(s,t;n',k)-\e},
    \end{equation} 
    for any $\e>0$.
    Then for admissible $(s,t;n,k)$, we know
    \begin{equation}
        \inf_{E:\ (s,t;n,k)\textup{-set}}\#E\gtrsim_\e p^{\bF(s,t;n,k)-\e},
    \end{equation} 
    for any $\e>0$.
    
\end{proposition}

We prove the following lemma. Actually, we also need to deal with a technical thing as mentioned before Lemma \ref{lempig1}. Since we have done it once in the proof of Lemma \ref{lempig1}, we will not reproduce here.
\begin{lemma}\label{lempig2}
Assume \eqref{conditionprop2} holds.
    Suppose $(s,t;n,k)$ is admissible and $n\ge k+2$. Then for any $(s,t;n,k)$-set $E$, there exist numbers $t_1,t_2,s_1,s_2$, such that
    \begin{equation}
        \begin{cases}
            1\le p^{t_1}\le p^{(k+1)(n-k-1)}, 1\le p^{t_2}\le p^{k+1};\\
            p^{s}\lesssim p^{s_1}\le p^{k}, 1\le p^{s_2}\le p;
        \end{cases}
    \end{equation}
    and
    \begin{equation}
        \begin{cases}
            p^{t_1+t_2}\gtrapprox p^t;\\
            p^{s_1+s_2}\gtrsim_\e p^{\bF(s,t_2;n-1,k)-\e};
        \end{cases}
    \end{equation}
    and
    \begin{equation}\label{lempig2ineq}
        \#E \gtrsim_\e p^{\bF(s_1,t_1;n-1,k)+s_2-\e}.
    \end{equation}
    
\end{lemma}

\begin{proof}
    We will do several steps of pigeonholing, which lead to the parameters $t_1,t_2,s_1,s_2,u,v$.

    By the definition of $(s,t;n,k)$-set, we write
    \[ E=\bigcup_{V\in\cV}Y(V), \]
    where $\cV\subset A(k,\Fp^{n})$ with $\#\cV\sim p^t$, and $Y(V)\subset V$ with $\#Y(V)\sim p^s$.

    The strategy is to find a subset of $\cV$, and a subset of $Y(V)$ for each $V$, still denoted by $\cV$, $Y(V)$, so that the new $\cV, Y(V)$ satisfy certain uniformity. This is done by a sequence of pigeonhole arguments.

    By properly choosing the coordinate, we can assume that $\gtrsim 1$ fraction of the $k$-planes in $\cV$ are not parallel to the $n$-th axis $\bzz^{n-1}\times \Fp$. Throwing away some $V\in\cV$, we can just assume all the $k$-planes in $\cV$ are not parallel to $\bzz^{n-1}\times \Fp$ with $\cV$ still satisfying $\#\cV\gtrsim p^t$. Because of this, we see that for any $V\in\cV$, the projection of $V$ onto $\Fp^{n-1}\times \bzz$, denoted by $\pi_{\Fp^{n-1}}(V)$ is a $k$-plane in $\Fp^{n-1}\times \bzz$.

    Now, we partition $\cV$ according to $A(k,\Fp^{n-1}\times \bzz)$. For each $W\in A(k,\Fp^{n-1}\times \bzz)$, we define $\cV_W$ to be the set of $V\in\cV$ such that $\pi_{\Fp^{n-1}}(V)=W$. We obtain a partition
    \[ \cV=\bigsqcup_{W\in A(k,\Fp^{n-1}\times \bzz)}\cV_W. \]

    \textit{First pigeonholing.} We assume that there exists $\cW\subset A(k,\Fp^{n-1}\times \bzz)$ with $\#\cW\sim p^{t_1}$ such that for each $W\in\cW$, $\#\cV_W\sim p^{t_2}$. Here, $1\le p^{t_1}\le p^{(k+1)(n-k-1)},1\le p^{t_2}\le p^{k+1}$ are dyadic numbers, and 
    \begin{equation}
        p^{t_1}p^{t_2}\gtrapprox p^t.
    \end{equation}
     The upper bound for $t_1,t_2$ is because $\#A(k,\Fp^{n-1}\times \bzz)\lesssim p^{(k+1)(n-k-1)}$, and $\#\cV_W\le \#A(k,W\times \Fp)\lesssim p^{k+1}$. We still use $\cV$ to denote $\bigcup_{W\in\cW} \cV_W$. And we see that
     \[ \#\cV\gtrapprox p^t. \]

    Fix a $W\in\cW$. We want to estimate $\bigcup_{V\in\cV_W} Y(V)(\subset W\times \Fp\cong \Fp^{k+1})$. This is some sort of Furstenberg set for $k$-planes in $(k+1)$-dimensional space. Also, an important property is that each $V\in\cV_W$ is not parallel to the last coordinate. Therefore, any line parallel to $\bzz^{n-1}\times \Fp$ will intersect $V$ at a single point. We will exploit some ``quasi-product" structure of this set. 

    \begin{lemma}\label{lems1s2}
        Assume \eqref{conditionprop2} holds for $n'=k+1$.
        Suppose $E=\bigcup_{V\in\cV} Y(V)$ is a $(s,t_2;k,k+1)$-set. In addition, we assume that each $V\in\cV$ intersects $\bzz^k\times \Fp$ at a single point. Then there exist dyadic numbers $p^{s_1},p^{s_2}$ such that
        \[ p^{s}\lesssim p^{s_1}\le p^k ,  \]
        \[ 1\le p^{s_2}\le p, \]
         and
        \[ p^{s_1}p^{s_2}\gtrsim_\e p^{\bF(s,t_2;k+1,k)-\e}, \]
        for any $\e>0$.
        In addition, $E$ contains a subset of form
        \[ \bigsqcup_{\xi\in\Xi}Y_\xi, \]
        where $\Xi\subset \Fp^{k}\times \bzz$ with $\#\Xi\sim p^{s_1}$, and $Y_\xi\subset E\cap\bigg(\xi+(\bzz^{k}\times \Fp)\bigg)$ with $\#Y_\xi\sim p^{s_2}$.
    \end{lemma}

    \begin{proof}
        For $\xi\in \Fp^{k}\times \bzz$, define $E(\xi):=E\cap\bigg(\xi+(\bzz^{k}\times \Fp)\bigg)$.
        By pigeonholing on vertical slices of $E$, we can find $\Xi\subset \Fp^{k}\times \bzz$ so that for all $\xi\in\Xi$, $E(\xi)$ have comparable cardinality. We may denote $\#\Xi\sim p^{s_1}, \#E(\xi)\sim p^{s_2}$ for two dyadic numbers $p^{s_1},p^{s_2}\ge 1$. We can also assume $p^{s_1} p^{s_2}\gtrapprox \#E$. Hence, assuming \eqref{conditionprop2} holds for $n'=k+1$, we have
        \[ p^{s_1} p^{s_2} \gtrsim_\e p^{\bF(s,t_2;k+1,k)-\e}. \]
        However, we still need the condition $p^{s_1}\gtrsim p^{s}$. This is handled by the same trick as in the proof of Lemma \ref{lemuv}. We will perform an algorithm to find a sequence $E=E_1\supset E_2\supset\cdots\supset E_m$, and disjoint subsets $\Xi_1,\Xi_2\cdots,\Xi_m$ of $\Fp^k\times \bzz$.
        To start with, let $E_1=E$. By pigeonholing, we can find $\Xi_1\subset \Fp^{k}\times \bzz$, such that $\#\Xi_1\sim p^{u_1}$, $\#E(\xi)\sim p^{v_1}$ for each $\xi\in\Xi$, and
        \[ p^{u_1}p^{v_1}\gtrsim_\e \#E_1\gtrsim_\e p^{\bF(s,t_2;2,1)-\e}. \]
        Suppose we have defined $E_{i}, \Xi_{i}$ for $1\le i\le j-1$. If $\sum_{i=1}^{j-1}p^{u_i}\ll p^{s}$, then we define
        \[ E_j:=E\setminus \bigcup_{1\le i\le j-1}\bigsqcup_{\xi\in\Xi_{i}}E(\xi). \]
        Since each $V\in\cV$ intersects $\xi+(\bzz^{k}\times \Fp)$ (and hence $E(\xi)$) at a single point, we see that $E_j$ is still a $(s,t_2;2,1)$-set. By the same pigeonholing, we find $\Xi_j\subset (\Fp^{k}\times \bzz)\setminus \bigcup_{1\le i\le j-1}\Xi_i$, so that $\#\Xi_j\sim p^{u_j}$, $\#E(\xi)\sim p^{v_j}$ for each $\xi\in\Xi_j$, and
        \[  p^{u_j}p^{v_j}\gtrsim_\e p^{\bF(s,t_2;2,1)-\e}. \]
        If $\sum_{i=1}^{j-1}p^{u_i}\gtrsim p^{s}$, then we stop and let $m=j-1$.

        The algorithm will stop in finite steps. Denote $\Xi:=\bigcup_{1\le i\le m}\Xi_i$. Let $p^{s_1},p^{s_2}$ be dyadic numbers so that $p^{s_2}=\min_{1\le i\le m}p^{v_i}$, $p^{s_1}\sim\#\Xi$. For each $\xi\in\Xi$, let $Y_\xi$ be a subset of $E(\xi)$ with $\#Y_\xi\sim p^{s_2}$. This is doable, since $\#E(\xi)\gtrsim p^{s_2}$ if $\xi\in\Xi$. One can check
        \[ p^{s_1} p^{s_2}\gtrsim_\e p^{\bF(s,t_2;2,1)-\e}. \]
        This is proved in the following way. Suppose $p^{s_2}=p^{v_i}$, then
        \[ p^{s_1}p^{s_2}=p^{s_1}p^{v_i}\ge p^{u_i}p^{v_i}\gtrsim_\e p^{\bF(s,t_1;2,1)-\e}. \]
        Also, one can check
        \[ p^{s_1}=\sum_{i=1}^mp^{u_i}\gtrsim p^{s}. \]

    \end{proof}

    Return to $ E_W:=\bigcup_{V\in\cV_W} Y(V)$ (which lies in $W\times \Fp\cong \Fp^{k+1}$). Recall $\#\cV_W\gtrapprox p^{t_2}$ and $\#Y(V)\sim p^{s} $.
    We see that $ E_W$ is a $(s,(t_2-\e)_+;k+1,k)$-set.

    By Lemma \ref{lems1s2}, there exist dyadic numbers $p^{s_1(W)}, p^{s_2(W)}$, so that $ E_W$ contains a subset of form
    \[ \bigsqcup_{\xi\in \Xi_W} Y_W(\rho_\xi). \]
Here, $\Xi_W\subset W(\subset\Fp^{n-1}\times \bzz)$ with $\#\Xi_W\sim p^{s_1(W)}$; $\rho_\xi:=\xi+(\bzz^{n-1}\times \Fp)$ and $Y_W(\rho_\xi)\subset \rho_\xi$ with $\#Y_W(\rho_\xi)\sim p^{s_2(W)}$. We also have
\[ p^{s_1(W)}\gtrsim p^{s}. \]
\[ p^{s_1(W)}p^{s_2(W)}\gtrsim_\e p^{\bF(s,(t_2-\e)_+;k+1,k)-\e}. \]
By pigeonholing on $(s_1(W),s_2(W))$ again, we can throw away some $W\in\cW$ so that the cardinality of $\cW$ is still $\gtrapprox p^{t_1}$, and assume that there exist $s_1,s_2$ such that
\[ (s_1(W),s_2(W))=(s_1,s_2),\ \textup{~for~all~}W\in\cW, \]
and
\[ p^{s_1}\gtrsim p^{s}. \]
\[ p^{s_1}p^{s_2}\gtrsim_\e p^{\bF(s,(t_2-\e)_+;2,1)-\e}. \]
By Lemma \ref{lipint}, we have
\[ p^{s_1}p^{s_2}\gtrsim_\e p^{\bF(s,t_2;2,1)-\e}. \]

Finally, we are about to estimate the lower bound of the cardinality of $E=\bigcup_{W\in\cW}\bigcup_{V\in\cV_W}Y(V).$ This $E$ is not the original one, but the one after pigeonholing. We write
\begin{equation}
    \begin{split}
        E=\bigcup_{W\in\cW}E_W&\supset\bigcup_{W\in\cW} \bigcup_{\xi\in\Xi_W}Y_W(\rho_\xi).
    \end{split}
\end{equation} 
Since $\#Y_W(\rho_\xi)\sim p^{s_2}$, we see that
\[ \#E\gtrsim \#(\bigcup_{W\in\cW}\Xi_W) \cdot p^{s_2}. \]
Since $\bigcup_{W\in\cW}\Xi_W$ is a $(s_1,t_1;n-1,k)$-set, by \eqref{conditionprop2}, we have
\[ \#E\gtrsim_\e p^{\bF(s_1,t_1;n-1,k)-\e}p^{s_2}. \]

\end{proof}

The next lemma is a recursive inequality of $\bF(s,t;n,k)$, which can be used to bound the right-hand side of \eqref{lempig2ineq}.
\begin{lemma}\label{recursionF2}
    Suppose $(s,t;n,k)$ is admissible and $n\ge k+2$.
    Suppose that there exist numbers $t_1,t_2,s_1$, such that
    \begin{equation}
        \begin{cases}
            t_1\in [0,(k+1)(n-k-1)];\\
            t_2\in [0,k+1];\\
            s_1\in [s,k];
        \end{cases}
    \end{equation}
    and
    \begin{equation}
            t_1+t_2= t;
    \end{equation}

    Then, 
    \begin{equation}
        \bF(s_1,t_1;n-1,k)+\max\{\bF(s,t_2;k+1,k)-s_1,0\}\ge \bF(s,t;n,k).
    \end{equation}

\end{lemma}

Let us quickly see how Lemma \ref{lempig2} and Lemma \ref{recursionF2} combined to prove Proposition \ref{proplower2}.
\begin{proof}[Proof of Proposition \ref{proplower2} assuming Lemma \ref{lempig2} and Lemma \ref{recursionF2}]\hfill

If $s=0$ or $t=0$, by Proposition \ref{s0} or \ref{t0}, we are done. So, we suppose $s,t>0$.

Fix $0<\e<\min\{\si/2,t/2, 1/2\}$, where $s=d+\si$ with $d\in\N, \si\in(0,1]$. We just need to prove
\begin{equation}
        \#E\gtrsim_\e p^{\bF(s,t;n,k)-C\e}
    \end{equation} 
for any $(s,t;n,k)$-set $E$. Here, $C$ is a large number.    

We only need to prove for large enough $p$, since for small $p$ we can choose the constant to be large. 

Suppose \eqref{conditionprop2} holds. If $E$ is a $(s,t;n,k)$-set, by Lemma \ref{lempig2}, there exist numbers $t_1,t_2,s_1,s_2$, such that
    \begin{equation}
        \begin{cases}
            t_1\in [0,(k+1)(n-k-1)], t_2\in [0,k+1];\\
            s_1\in [s-\e,k], s_2\in [0,1];
        \end{cases}
    \end{equation}
    and
    \begin{equation}
        \begin{cases}
            t_1+t_2\ge t-\e;\\
            s_1+s_2\ge \bF(s,t_2;k+1,k)-\e;
        \end{cases}
    \end{equation}
    and
    \begin{equation}
        \#E \gtrsim_\e p^{\bF(s_1,t_1;n-1,k)+s_2-\e}\ge p^{\bF(s_1,t_1;n-1,k)+\max\{\bF(s,t_2;k+1,k)-s_1,0\}-2\e}.
    \end{equation}
    
Let $t'=t-\e$ which is $>0$ by assumption.
Choose $0\le t_i'\le t_i$ for $i=1,2$, so that $t_1'+t_2'=t'$. Choose $s_1\le s_1'\le s_1+\e$, so that $s_1'\in [s,k]$. Then $t',t_1',t_2',s_1'$ satisfy the condition in Lemma \ref{recursionF2}, we obtain
\[ \bF(s_1',t_1';n-1,k)+\max\{\bF(s,t_2';k+1,k)-s_1,0\}\ge \bF(s,t';n,k). \]
By the monotonicity of $t\mapsto\bF(s,t;n,k)$, Lemma \ref{lipint} and Lemma \ref{lipins}, we have
\[ \#E\gtrsim_\e p^{\bF(s,t;n,k)-C\e}. \]

\end{proof}

It remains to prove Lemma \ref{recursionF2}.
\begin{proof}[Proof of Lemma \ref{recursionF2}]\hfill

    When $s=0$, we have  $\bF(0,t;n,k)=\max\{0,t-k(n-k)\}$. By Lemma \ref{easybound}, we have
    \[ \bF(s_1,t_1;n-1,k)+\bF(0,t_2;k+1,k)-s_1\ge s_1+t_1-k(n-k-1)+t_2-k-s_1= t-k(n-k). \]

    When $s>0$, write $s=d+\si$. If also $t\le (k-d-1)(n-k)$, then
    \[ \bF(s,t;n,k)=s. \]
    By Lemma \ref{easybound},
    \[ \bF(s_1,t_1;n-1,k)+\bF(s,t_2;k+1,k)-s_1\ge s_1+s-s_1=s. \]
    
    So we consider $s>0$ and $t>(k-d-1)(n-k)$.
    Denote
    \[ \bF:=\bF(s_1,t_1;n-1,k)+\max\{\bF(s,t_2;k+1,k)-s_1,0\}. \]
    We have both 
    \begin{equation}\label{induction2}
        \bF\ge \bF(s,t_2;k+1,k)+\bF(s_1,t_1;n-1,k)-s_1.
    \end{equation}
    and
    \begin{equation}
        \bF\ge \bF(s_1,t_1;n-1,k).
    \end{equation}

    Write the canonical expression
    \begin{equation}
    \begin{cases}
        s=d+\si\\
        t=(k-d-1)(n-k)+(d+2)m+\tau,
    \end{cases}
\end{equation} 
    where $d\in\{0,\dots,k-1\}$, $\si\in (0,1]$,  $m\in\{0,\dots,n-k-1\}$, $\tau\in (0,d+2]$. We have
    \[ \bF(s,t;n,k)=s+m+\min\{\tau,\frac12(\si+\tau),1\}. \]

Next, we analyze $\bF(s,t_2;k+1,k)$.
\medskip

    \fbox{Case (1): $t_2\le k-d-1$} We use
    \[ \bF\ge \bF(s_1,t_1;n-1,k) \ge \bF(s,t_1;n-1,k). \]
Since $t_2\le k-d-1$, we have
\[ t_1\ge (k-d-1)(n-k-1)+(d+2)m+\tau. \]

Therefore,
\[ \bF\ge\bF(s,(k-d-1)(n-k-1)+(d+2)m+\tau;n-1,k)= s+m+\min\{ \tau,\frac12(\si+\tau), 1 \}. \]
Therefore,
\[ \bF\ge\bF(s,t;n,k). \]

\fbox{Case (2): $t_2>k-d-1$} We write $t_2=k-d-1+\tau_2$, where $\tau_2\in (0,d+2]$. 
We have
\[ \bF(s,t_2;k+1,k)= s+\min\{ \tau_2,\frac12(\si+\tau_2),1 \}. \]

Since $s_1\ge s$, write 

\[s_1=d+\si_1,\]
with $\si_1\ge \si$.

\fbox{Subcase (2.1): $\si_1\in [\si,1]$}  We have
\[ t_1=t-t_2= (k-d-1)(n-k-1)+(d+2)m+\tau-\tau_2. \]

\fbox{Subsubcase (2.1.1): $\tau-\tau_2>0$} We have the canonical expression for $(s_1,t_1)$:
\begin{equation*}
    \begin{cases}
        s_1=d+\si_1,\\
        t_1=t-t_2= (k-d-1)(n-k-1)+(d+2)m+\tau-\tau_2.
    \end{cases}
\end{equation*}
Therefore,
\[ \bF(s_1,t_1;n-1,k)=s_1+m+\min\{ \tau-\tau_2,\frac12(\si_1+\tau-\tau_2),1 \}. \]
Therefore,
\[ \bF\ge s+m+\min\{ \tau_2,\frac12(\si+\tau_2),1 \}+\min\{\tau-\tau_2,\frac{1}{2}(\si+\tau-\tau_2),1\}. \]
We just need to show
\[ \min\{ \tau_2 ,\frac12(\si+\tau_2),1 \}+\min\{\tau-\tau_2 ,\frac{1}{2}(\si+\tau-\tau_2 ),1\}\ge \min\{\tau ,\frac12(\si+\tau ),1\}. \]
To show $\min_i a_i+\min_j b_j\ge \min_k c_k$, it suffices to show $\forall i, j, \exists k$, s.t. $a_i+b_j\ge c_k$. We can check:

$\bullet$ $\tau_2+\tau-\tau_2\ge \tau$.

$\bullet$ $\tau_2+\frac12(\si+\tau-\tau_2)\ge \frac12(\si+\tau)$.

$\bullet$ $\frac12(\si+\tau_2)+\tau-\tau_2\ge \frac12(\si+\tau).$

$\bullet$ $\frac12(\si+\tau_2)+\frac12(\si+\tau-\tau_2)\ge \frac12(\si+\tau)$.

\medskip

\fbox{Subsubcase (2.1.2): $\tau-\tau_2\le 0, m\ge 1$} We write $t_1=(k-d-1)(n-k-1)+(d+2)(m-1)+d+2+\tau-\tau_2$. We have
\[\bF(s_1,t_1;n-1,k)=s_1+m-1+\min\{\tau-\tau_2+d+2,\frac12(\si_1+\tau-\tau_2+d+2),1\}.\]
Therefore,
\[ \bF\ge s+m-1+\min\{\tau_2,\frac12(\si+\tau_2),1\}+\min\{\tau-\tau_2+d+2,\frac12(\si_1+\tau-\tau_2+d+2 ),1\}. \]
We just need to prove
\[ \min\{\tau_2,\frac12(\si+\tau_2),1\}+\min\{\tau-\tau_2+d+1,\frac12(\si+\tau-\tau_2+d ),0\}\ge \min\{\tau,\frac12(\si+\tau),1\}. \]
This is not hard to verify by noting $\tau\le\tau_2\le d+2$.

$\bullet$ $\tau_2+\tau-\tau_2+d+1\ge \tau$.

$\bullet$ $\tau_2+\frac12(\si+\tau-\tau_2+d)\ge \frac12(\si+\tau)$.

$\bullet$ $\tau_2+0\ge \tau$.

$\bullet$ $\frac12(\si+\tau_2)+\tau-\tau_2+d+1\ge \tau+d+1-\frac12\tau_2\ge \tau$.

$\bullet$ $\frac12(\si+\tau_2)+\frac12(\si+\tau-\tau_2+d)\ge \frac12(\si+\tau)$.

$\bullet$ $\frac12(\si+\tau_2)+0\ge \frac12(\si+\tau)$.

$\bullet$ $1+\tau-\tau_2+d+1\ge \tau$.

$\bullet$ $1+\frac12(\si+\tau-\tau_2+d)\ge \frac12(\si+\tau)$.

$\bullet$ $1+0\ge 1$.

\medskip
\fbox{Subsubcase (2.1.3): $\tau-\tau_2\le 0, m=0$} We have $t_1\le (k-d-1)(n-k-1)$, so 
\[ \bF(s_1,t_1;n-1,k)=s_1. \]
Therefore,
\[ \bF\ge \bF(s,t_2;k+1,k)= s+\min\{\tau_2,\frac12(\si+\tau_2),1\}. \]
Also, $\bF(s,t;n,k)=s+\min\{\tau,\frac12(\si+\tau),1\}$.
We see that $\bF\ge \bF(s,t;n,k)$ since $\tau_2\ge \tau$.

\fbox{Subcase(2.2): $\si_1>1$} Write $s_1=(d+\De)+(\si_1-\De),$ where $\De\in \N^+$ so that $\si_1-\De\in (0,1]$. We have
\begin{align*}
    t_1 &= (k-d-1)(n-k-1)+(d+2)m+\tau-\tau_2\\
    &= (k-d-\De-1)(n-k-1)+(d+2)m+\tau-\tau_2+\De(n-k-1)\\
    &=(k-d-\De-1)(n-k-1)+(d+2)m_1+\tau_1.
\end{align*} 
Here, the last line is the canonical expression for $t_1$ (see \eqref{expression}), in which $0\le m_1\le n-k-2$, $\tau_1\in (0,d+\De+2]$.

If $(d+2)m+\tau-\tau_2\le 0$, then $m=0, \tau_2\ge \tau$. By the same reasoning as in Subsubcase (2.1.3), we have $\bF\ge \bF(s,t;n,k)$.

Hence, we can assume $(d+2)m+\tau-\tau_2> 0$
We have
\[ \bF(s_1,t_1;n-1,k)= s_1+m_1+\min\{ \tau_1 ,\frac12(\si_1-\De+\tau_1 ),1 \}. \]

Since $(d+2)m+\tau-\tau_2\le(d+2)m_1+\tau_1$ and $\tau,\tau_1,\tau_2\in (0,d+2]$, we have $m_1\ge m-1$.
If $m_1\ge m+1$, then $\bF\ge \bF(s_1,t_1;n-1,k)\ge s+m+1\ge \bF(s,t;n,k)$.
We only need to consider two cases: $m_1=m, m_1=m-1$.

We have
\[ \bF\ge s+\min\{ \tau_2 ,\frac12(\si+\tau_2 ),1 \}+m_1+\min\{ \tau_1,\frac12(\si_1-\De+\tau_1 ),1 \}. \]

\fbox{Subsubcase (2.2.1): When $m_1=m,\tau_1=\tau-\tau_2+\De(n-k-1)$} We have
\[ \bF\ge s+m+\min\{ \tau_2 ,\frac12(\si+\tau_2 ),1 \}+\min\{(\tau-\tau_2)\vee 0 ,\frac{1}{2}(\si+\tau-\tau_2 ),1\}, \]
where we used $\tau_1\ge (\tau-\tau_2)\vee 0,$ and $\si_1-\De+\tau_1\ge \si-\De+\tau-\tau_2+\De(n-k-1)\ge \si+\tau-\tau_2$.

We just need to check
\[ \min\{ \tau_2 ,\frac12(\si+\tau_2 ),1 \}+\min\{(\tau-\tau_2)\vee 0 ,\frac{1}{2}(\si+\tau-\tau_2),1\}\ge \min\{\tau ,\frac12(\si+\tau ),1\}. \]
A less obvious case will be $\frac12(\si+\tau_2)+(\tau-\tau_2)\vee0\ge \frac12(\si+\tau_2)+\frac12(\tau-\tau_2)\ge \frac12(\si+\tau)$. We leave out the verification of other cases.

\medskip

\fbox{Subsubcase (2.2.2): When $m_1=m-1, \tau_1=\tau-\tau_2+d+2+\De(n-k-1)$} Since $n\ge k+2$, we have
\[ \bF\ge s+m-1+\min\{\tau_2  ,\frac12(\si+\tau_2  ),1\}+\min\{\tau-\tau_2+d+2+\De  ,\frac12(\si+\tau-\tau_2+d+2 ),1\}. \]
Since $\tau_1\le d+2$, we have $\tau_2\ge \tau$.
Note that 
\[\min\{\tau-\tau_2+d+2+\De  ,\frac12(\si+\tau-\tau_2+d+2 ),1\}-1\ge \min\{(\tau-\tau_2+1)\vee 0  ,\frac12(\si+\tau-\tau_2 ),0\}.\]

We just need to prove
\[ \min\{\tau_2  ,\frac12(\si+\tau_2  ),1\}+\min\{(\tau-\tau_2+1)\vee 0  ,\frac12(\si+\tau-\tau_2 ),0\}\ge \min\{\tau  ,\frac12(\si+\tau),1\}. \]
This is not hard to verify by noting $\tau_2\ge\tau$. A less obvious case is $\frac12(\si+\tau_2)+(\tau-\tau_2+1)\vee 0\ge \frac12(\si+\tau_2)+\frac12(\tau-\tau_2+1)\ge \frac12(\si+\tau)$.

\end{proof}

\begin{proof}[Proof of Theorem \ref{lower bound}]
    Finally, we see that Theorem \ref{lower bound} can be obtained from Proposition \ref{proplower1} and Proposition \ref{proplower2} using induction on $(n,k)$.
\end{proof}

\

\section{Exceptional set estimate: lower bound}\label{sec6}

We prove Theorem \ref{exceptionallowerthm} in this section.
The result in this section has been studied by the author \cite{gan2024exceptional} under the Euclidean setting, and the statement there is slightly different from this paper. We still provide the details for the finite-field setting, though they are essentially the same.
 
\begin{proof}[Proof of Theorem \ref{exceptionallowerthm}]
Recall that $(n,k)=(2,1)$ and $\frac{a}{2}<s\le \min\{1,a\}$ was considered in Proposition \ref{oberlin2d}. We denote $\bA_{s,a}\subset \Fp^2$ to be the set constructed there which satisfies $\#\bA_{s,a}\gtrsim p^a$, and
\[ E_s(\bA_{s,a};2,1)\gtrsim p^{2s-a}.\]

\medskip

Next, we prove for general $(n,k)$. Note that Type 1 and Type 4 are easy, so we consider Type 2 and Type 3. Recall the canonical expression as in Definition \ref{defM}: $a=m+\beta,s=l+\ga$.

\medskip

\fbox{Type 2} We have $l+1\le m\le n+l-k$. For $\ga$, we only use $\ga>\beta$. We first show 
the existence of $A$ with $\#A\sim p^a$ such that $\#E_s(A;n,k)\gtrsim p^{k(n-k)-(m-l)(k-l)}$. 

Choose $A=\Fp^m\times I\times \bzz^{n-m-1}$, where $I\subset \Fp$ has cardinality $p^\beta$. For simplicity, we just write $A=\Fp^m\times I$. We want to find those $V\in G(n-k,\Fp^n)$ such that $\#(\pi_V^*(A))< p^{l+\ga}$. We note that if $V$ satisfies $\#\pi_V^*(\Fp^m)\le p^l$, then $\#\pi_V^*(\Fp^m\times I)\le p^{l+\beta}<p^{l+\ga}$, which means $V\in E_s(A;n,k)$. This shows that
\[ E_s(A;n,k)\supset \{V\in G(n-k,n): \#\pi_V^*(\Fp^m)\le p^l\}. \]
By Lemma \ref{easylem2}, the right hand side has cardinality $\sim p^{k(n-k)-(k-l)(m-l)}$.

\medskip

\fbox {Type 3, $\ga\in(\frac{\beta}{2},\beta]$} We have $n-m-1\ge k-l, \ga> \frac{\beta}{2}$. We show 
the existence of $A$ with $\#A\sim p^a$ such that $\#E_s(A;n,k)\gtrsim p^{k(n-k)-(m+1-l)(k-l)+2\ga-\beta}$.

We choose $A=\bA_{\ga,\beta}\times \Fp^m\times \bzz^{n-m-2}$, where $\bA_{\ga,\beta}\subset \Fp^2$ is such that $\#\bA_{\ga,\beta}\sim p^\beta$ and $\#E_\ga(\bA_{\ga,\beta};2,1)\gtrsim p^{2\ga-\beta}$. We will use $\Fp^2$ to denote $\Fp^2\times \bzz^{n-2}$. We use $\Fp^m$ to denote $\bzz^2\times \Fp^m\times \bzz^{n-m-2}$. We use $\theta$ to denote the lines in $G(1,\Fp^2)$. Define
\[ \cU_\theta=\{ V\in G(n-k,\Fp^n): \#(\pi^*_V(\theta \oplus \Fp^m))\le p^l, \#(\pi_V^*(\Fp^m))=p^l, \#(\pi_V^*(\Fp^2\times \Fp^m))=p^{l+1} \}, \]
\[\cV_\theta=\{ V\in G(n-k,\Fp^n): \#(\pi_V^*(\theta \oplus \Fp^m))\le p^l\},\]
\[\cW_\theta=\{V\in \cV_\theta: \#(\pi_V^*(\Fp^m))\le p^{l-1}\}\cup \{V\in \cV_\theta: \#(\pi_V^*(\Fp^2\times \Fp^m)\le p^l \}.\]
It is not hard to see $\cU_\theta=\cV_\theta\setminus \cW_\theta$. 

Let $\bE_{\ga,\beta}:=E_\ga(\bA_{\ga,\beta};2,1)=\{ \theta\in G(1,\Fp^2): \#(\pi_{\theta}^*(\bA_{\ga,\beta}))<p^\ga \}$. We claim that if $\theta\in \bE_{\ga,\beta}$, then $\cV_\theta\subset E_{l+\ga}(A)$. In other words, $V\in \cV_\theta$ satisfies $\#(\pi_V^*(A))<p^{l+\ga}$. We first note that by definition $\bA_{\ga,\beta}\subset \bigcup_{L\in\pi_{\theta}^*(\bA_{\ga,\beta})}L$, so
\[ A=\bA_{\ga,\beta}\times \Fp^m\subset (\bigcup_{L\in\pi_{\theta}^*(\bA_{\ga,\beta})}L)\times \Fp^m= (\bigcup_{L\in\pi_{\theta}^*(\bA_{\ga,\beta})}L\times \Fp^m). \]
Therefore,
\[ \#(\pi_V^*(A))\le \#(\pi_{\theta}^*(\bA_{\ga,\beta}))\cdot\#(\pi_V^*(\theta\oplus\Fp^m))< p^{\ga+l}. \]

We see that $E_{l+\ga}(A;n,k)\supset \bigcup_{\theta\in \bE_{\ga,\beta}} \cV_\theta.$
$\cV_\theta$ may intersect with each other, so we delete a tiny portion $\cW_\theta$ from $\cV_\theta$, so that the remaining part $\cU_\theta=\cV_\theta\setminus \cW_\theta$ are disjoint with each other. Then we have
\[ E_{l+\ga}(A;n,k)\supset \bigsqcup_{\theta\in\bE_{\ga,\beta}}\cU_\theta. \]
To make the argument precise, we first apply Lemma \ref{easylem2} to see \[\#\cV_\theta\sim p^{k(n-k)-(m+1-l)(k-l)}.\]
To calculate $\#\cW_\theta$, we apply Lemma \ref{easylem2} again to see
\[ \#\{V\in G(n-k,n):\dim(\pi_V^*(\Fp^m))\le p^{l-1}\}\sim p^{k(n-k)-(m-l+1)(k-l+1)}, \]
and
\[ \#\{V\in G(n-k,n):\dim(\pi_V^*(\Fp^2\times \Fp^m))\le p^l\}\sim p^{k(n-k)-(m+2-l)(k-l)}. \]
Therefore, $\#\cW_\theta\ll \#\cV_\theta$ (when $p$ is large enough) and hence
\[ \#\cU_\theta\sim\#\cV_\theta\sim p^{k(n-k)-(m+1-l)(k-l)}. \]

Next we show $\cU_\theta\cap \cU_{\theta'}=\emptyset$ for $\theta\neq\theta'$.
By contradiction, suppose $V\in \cU_\theta\cap \cU_{\theta'}$. By definition, $V$ satisfies
\begin{align*}
\#(\pi_V^*(\theta \oplus \Fp^m))\le p^l,\\
\#(\pi_V^*(\theta' \oplus \Fp^m))\le p^l,\\
\#(\pi_V^*(\Fp^m))=p^l,\\
\#(\pi_V^*(\Fp^2\times \Fp^m))=p^{l+1}. 
\end{align*}
The first and third conditions imply $\pi_V^*(\theta)\subset \pi_V^*(\Fp^m)$. Similarly, the second and third conditions imply $\pi_V^*(\theta')\subset \pi_V^*(\Fp^m)$. Noting that $\textup{span}\{\theta,\theta'\}=\Fp^2$, we have 
\[ \pi_V^*(\Fp^2\times \Fp^m)\subset \pi_V^*(\Fp^m), \]
which has cardinality $=p^l$, contradicting the fourth condition.

We have shown that $E_{l+\ga}(A;n,k)\supset \bigsqcup_{\theta\in \bE_{\ga,\beta}}\cU_\theta$, which has cardinality
$\#\bE_{\ga,\beta}\cdot\#\cU_\theta$
\[ \gtrsim p^{2\ga-\beta+k(n-k)-(m+1-l)(k-l)}. \]

\medskip

\fbox{Type 2, $\ga\in (\frac{\beta+1}{2},1]$} 
The trick is to write $a=(m-1)+(\beta+1)=:m'+\beta'$. Note that $l\le m'\le n+l-k-1$ and $\ga>\frac{\beta'}{2}$. We apply the construction in ``Type 3, $\ga\in (\frac{\beta}{2},\beta]$", where we only used $\ga>\frac{\beta}{2}$. We obtain a $A$ with $\#A\sim p^a$ such that
\[ \#E_s(A;n,k)\gtrsim p^{k(n-k)-(m'+1-l)(k-l)+2\ga-\beta'}= p^{k(n-k)-(m-l)(k-l)+2\ga-(\beta+1)}. \]

\medskip

\fbox{Type 3, $\ga\in (0,\frac\beta2]$} The trick is to make $a$ bigger. We write $a=m+\beta\le (m+1)+0=:m'+\beta'$. One can check $l+1\le m'\le n+l-k$ and $\ga>\beta'$. We apply the construction in \fbox{Type 2} with $(m',\beta')$. We find $A'$ with $\#A'\sim p^{m+1}$ such that
\[ \#E_s(A';n,k)\gtrsim p^{k(n-k)-(m+1-l)(k-l)}. \]
Choosing any subset $A\subset A'$ with $\#A\sim p^a$, we get
\[ \#E_s(A;n,k)\ge\#E_s(A';n,k)\gtrsim p^{k(n-k)-(m+1-l)(k-l)}. \]

\end{proof}

\

\section{Exceptional set estimate: Upper bound I}\label{sec7}

We prove the following proposition in this section.
\begin{proposition}\label{exprop1}
    Assume Conjecture \ref{conj2d} holds.
    For $0<a\le k+1$ and $0<s$, and $A\subset \Fp^{k+1}$ with $\#A\gtrsim p^a$, we have
\begin{equation}
    \#E_{s-\e}(A;k+1,k)\lesssim_\e p^{\e+\bM(a,s;k+1,k)},
\end{equation}
for any $\e>0$.
\end{proposition}

It is proved by using the following lemma.

\begin{lemma}\label{refine}
Assume Conjecture \ref{conj2d} holds.
Fix $0<\e_0<1/100$. Let $0<a\le k+1, s>0$ and $\si\ge 0$. Suppose that the parameters
satisfy $\bM(a,s;k+1,k)+\e_0\le k$ and $a-s+\e_0\le \si\le a$. Let $A\subset \Fp^{k+1}$ with $\#A=p^{k+1}$. Define
\[ \cV:=\{ L\in A(1,\Fp^{k+1}):p^{\si}\le \#(L\cap A)  \}. \]
Then 
\[ \#\cV\lesssim_{\e_0} p^{\bM(a,s;k+1,k)+a-\si+\e_0}. \]
\end{lemma}

\begin{remark}
    {\rm
    The implicit constant in $\lesssim_{\e_0}$ can be made independent of $\si$, since we can apply the lemma to $\si\in \e_0^2\cdot\N$ which consists of $O(\e_0^{-2})$ values.
    }
\end{remark}

\begin{proof}
    The proof of this part is by a Furstenberg estimate. As in Definition \ref{defM}, write
    \begin{equation}
        \begin{cases}
            a=m+\beta\\
            s=l+\gamma.
        \end{cases}
    \end{equation}
There are four cases.

Type 1: When $s>\min\{a,k\}$,
\[ \bM(a,s;k+1,k)=k. \]

Type 2: When $s\le\min\{a,k\}$ and $m=l+1, \ga\in(\beta,1]$,
\begin{align*}
    \bM(a,s;k+1,k)=l+\max\{2\ga-(\beta+1),0\}.
\end{align*}

Type 3: When $s\le\min\{a,k\}$ and $m=l, \ga\in (0,\beta]$, 
\begin{align*}
\bM(a,s;k+1,k)=l+\max\{2\ga-\beta,0\}.
\end{align*}

Type 4: When $s\le a-1$,
\[ \bM(a,s;k+1,k)=-\infty. \]

Type 1 is trivial, since $\#\cV\le p^k$ is always true. Type 4 is also easy, since in this case we have $\si\ge a-s+\e_0\ge 1+\e_0$, and hence $\cV=\emptyset$. We only need to consider Type 2 and Type 3.

Suppose $\#\cV=p^t$, we want to show 
\[t\le \bM(a,s;k+1,k)+a-\si+\e_0.\]
We write $t=2w+\tau$, where $w\in \{0,\dots,k-1\}$, $\tau\in (0,2]$.

For convenience, we define 
\begin{equation}
R(\beta,\ga):=\begin{cases}
    \max\{2\ga-(\beta+1),0\} & m=l+1, \ga\in(\beta,1]\\
    \max\{2\ga-\beta,0\} & m=l, \ga\in (0,\beta].
\end{cases}
\end{equation}
Then $\bM(a,s;k+1,k)=l+R(\beta,\ga)$ in the case of Type 2 or Type 3.

\medskip 

Note that $A$ contains a $(\si,t;k+1,1)$-Furstenberg set $\bigcup_{L\in\cV}(L\cap A)$.
Since Conjecture \ref{conj2d} holds, by Theorem \ref{lower bound}, we have the lower bound:

\begin{equation}\label{fromthis}
    a+\eta_0\ge \bF(\si,t;k+1,1)=w+\min\{\si+\tau,\frac32\si+\frac12\tau,\si+1\}, 
\end{equation} 
for any $\eta_0>0$ when $p$ is big enough depending on $\eta_0$. We choose $\eta_0$ such that $\eta_0<\e_0/2$

From \eqref{fromthis}, there are three possible cases.

\medskip
(1) $a+\eta_0\ge w+\si+\tau$.

This implies $t=2w+\tau\le a+w-\si+\eta_0$. Our goal is to show $t\le \bM(a,s;k+1,k)+a-\si+\e_0$, so we just need to show 
\[ \bM(a,s;k+1,k)\ge w. \]
We also know 
\[ a\ge w+\si+\tau-\eta_0\ge w+\tau+a-s+\e_0-\eta_0\ge w+\tau+a-s+\e_0/2, \]
which implies 
\[ s\ge w+\tau+\e_0/2. \]
Since $s\le l+1$ and $w\in \N$, we have $w\le l$. Therefore, with the condition $\bM(a,s;k+1,k)\ge l$, we see that
\[ \bM(a,s;k+1,k)\ge  w. \]

(2) $a+\eta_0\ge \si+1+w$.

In this case, $a\ge \si+1+w-\eta_0\ge a-s+w+1+\e_0-\eta_0$, which implies
\[ s\ge w+1+\e_0/2. \]
Since $s\le l+1$ and $w\in\N$, we get $w\le l-1$.

Now we want to prove $t\le \bM(a,s;k+1,k)+a-\si+\e_0$. If this is not true, then by contradiction, we assume
\[ 2w+\tau=t> \bM(a,s;k+1,k)+a-\si+\e_0. \]
This will imply
\[ l-1+\tau\ge w+\tau>\bM(a,s;k+1,k)+a-\si-w+\e_0\ge \bM(a,s;k+1,k)+1+\e_0-\eta_0\ge l+1, \]
which is a contradiction, as $\tau\le 2$.

(3) $a+\eta_0\ge w+\frac32\si+\frac12\tau$.

This is equivalent to $t\le 2a-3\si+2\eta_0$. 
We hope to prove
\[t\le  \bM(a,s;k+1,k)+a-\si+\e_0.\]
We only need to prove
\[ 2a-3\si\le \bM(a,s;k+1,k)+a-\si, \]
since $2\eta_0<\e_0$.

This is equivalent to
\[ a\le \bM(a,s;k+1,k)+2\si. \]
We plug in $a=m+\beta,s=l+\ga, \bM(a,s;k+1,k)=l+R(\beta,\ga), \si\ge (m+\beta)-(l+\ga)+\e_0$. Then it suffices to prove
\[ m+\beta\le l+R(\beta,\ga)+2(m+\beta)-2(l+\ga)+2\e_0. \]
This inequality is equivalent to
\[ 2\ga-(\beta+m-l)\le R(\beta,\ga)+2\e_0, \]
This is easy to verify by checking either in the case $m=l+1,\ga\in(\beta,1]$ or in the case $m=l,\ga\in (0,\beta]$.
\end{proof}

\begin{proof}[Proof of Proposition \ref{exprop1}]

We will use the lemma that we just proved. We show that if $A\subset \Fp^{k+1}$ with $\#A\gtrsim p^a$, $E\subset G(1,\Fp^{k+1})$ with $\#E>C_\e p^{\e+\bM(a,s;k+1,k)}$, and $p$ is big enough, then there exists $V\in E$ such that
\[ \#\pi_V^*(A)>p^{s-\e}. \]

Since $\#E\le p^k$, we have
\[ \e+\bM(a,s;k+1,k)\le k. \]

For each $V\in E$ and a dyadic number $1\le \mu\le p$, define 
\[ \L_{V,\mu}:=\{ L\in A(1,\Fp^{k+1}): L\parallel V, \mu\le \#(A\cap L)< 2\mu  \}. \]
We see that $\{\L_{V,\mu}\}_\mu$ form a partition of the lines parallel to $V$.

We choose $\e_0\ll \e$, for example, $\e_0=\e^2$.
We will estimate the cardinality of the set $\bigcup_{V\in E} \L_{V,\mu}$ for $\mu\ge p^{a-s+\e_0}$. We write $\mu=p^\si$ with $a\ge\si\ge a-s+\e_0$. By Lemma \ref{refine}, we have 
\[ \#\bigcup_{V\in E} \L_{V,\mu}\lesssim_\e p^{\bM(a,s;k+1,k)+a-\si+\e_0}. \]

If we define
\[ E_\mu:=\{ V\in E: \#\L_{V,\mu}>p^{a-\si-\e_0} \}, \]
then
\[ \#E_\mu\lesssim_\e p^{\bM(a,s;k+1,k)+2\e_0}. \]
Summing over dyadic $\mu\ge p^{a-s+\e_0}$, we have
\[ \#\bigcup_{\mu\ge p^{a-s+\e_0}}E_\mu\lesssim_\e p^{\bM(a,s;k+1,k)+3\e_0}. \]
Therefore, $\#E>\#\bigcup_{\mu\ge p^{a-s+\e_0}}E_\mu$ when $p$ is large enough. We can find $V\in E\setminus \bigcup_{\mu\ge p^{a-s+\e_0}}E_\mu$. In other words, there exists $V\in E$, such that for any $\mu=p^\si\ge p^{a-s+\e_0}$,
\[ \#\L_{V,\mu}\le p^{a-\si-\e_0}. \]

Now, we do the estimate. We have
\[ \#A=\sum_{\mu<p^{a-s+\e_0}}\sum_{L\in \L_{V,\mu}}\#(A\cap L)+\sum_{\mu\ge p^{a-s+\e_0}}\sum_{L\in \L_{V,\mu}}\#(A\cap L). \]
Note that
\[ \sum_{\mu\ge p^{a-s+\e_0}}\sum_{L\in \L_{V,\mu}}\#(A\cap L)\le \sum_{\mu\ge p^{a-s+\e_0}}2\mu \cdot \#\L_{V,\mu}\le 2\sum_{\mu\ge p^{a-s+\e_0}}p^{a-\e_0}\le \frac12 \#A.  \]
In the last inequality, we are summing over $O(\log p)$ terms and we assume $p$ is big enough.

We have 
\[\frac12\#A\le \sum_{\mu<p^{a-s+\e_0}}\sum_{L\in \L_{V,\mu}}\#(A\cap L)<2p^{a-s+\e_0} \#\pi_V^*(A).\]
This implies $\#\pi_V^*(A)\gtrsim p^{s-\e_0}>p^{s-\e}$.

\end{proof}

\

\section{Exceptional set estimate: Upper bound II}\label{sec8}

The goal of this section is to prove the following proposition.
\begin{proposition}\label{expropupper2}
Let $n,k\in\N$ and $n\ge k+2$. Suppose for any $n'=k+1,\dots,n-1$ and $a\in (0,n'], s>0$, we have
    \begin{equation}\label{excondition}
        \sup_{A\subset \Fp^{n'}, \#A\sim p^a}\#E_{s-\e}(A;n',k)\lesssim_\e p^{\e+\bM(a,s;n',k)},
    \end{equation}
for any $\e>0$.
Then for $a\in (0,n], s>0$, we have
\[\sup_{A\subset \Fp^{n}, \#A\sim p^a}\#E_{s-\e}(A;n,k)\lesssim_\e p^{\e+\bM(a,s;n,k)},\]
for any $\e>0$.

\end{proposition}

We prove the following lemma.

\begin{lemma}\label{leminducM}
    Assume \eqref{excondition} holds. Suppose $n,k\in\N$, $n\ge k+2$, $a\in (0,n]$, $s>0$ and $s\in (a-(n-k),\min\{a,k\}]$. Then for any $A\subset \Fp^n$ with $\#A\sim p^a$, there exist numbers $a_1,s_1$, such that
    \begin{equation}
        \begin{cases}
            p^{\max\{0,a-1\}}\le p^{a_1}\le p^{\min\{n-1,a\}};\\
            1\le p^{s_1}\le p^s;
        \end{cases}
    \end{equation}
    and
    \begin{equation}
        \#E_{s-\e}(A;n,k)\lesssim_\e p^{\e+ \bM(a_1,s_1;n-1,k)+\bM(s_1+a-a_1,s;k+1,k)}.
    \end{equation}
\end{lemma}

\begin{proof}
    We will do several steps of pigeonholing, which lead to the parameters $s_1,a_1$. For convenience, we denote $\cV:=E_{s-\e}(A;n,k)$.

By properly choosing the coordinate, we can assume $\gtrsim 1$ fraction of the $(n-k)$-planes in $\cV$ intersect $\Fp^{n-1}\times \bzz$ transversely (at a $(n-k-1)$-plane). Since we are going to find upper bound for $\#\cV$, we can just assume all the planes in $\cV$ intersect $\Fp^{n-1}\times \bzz$ at a $(n-k-1)$-plane.

For each $W\in G(n-k-1,\Fp^{n-1}\times \bzz)$, we define $\cV_W:=\{V\in\cV: V\cap(\Fp^{n-1}\times\bzz)=W\}$. By pigeonholing, there exits $\cW\subset G(n-k-1,\Fp^{n-1}\times \bzz)$ with $\#\cW\sim p^{t_1}$ such that for each $W\in\cW$, $\#\cV_W\sim p^{t_2}$. Here, $p^{t_1},p^{t_2}\ge 1$ are dyadic numbers so that
\[ p^{t_1+t_2}\gtrapprox \#\cV. \]

By pigeonholing on $\#A\cap(\Fp^{n-1}\times \xi)$ for $\xi\in\Fp$, we can find a dyadic number $1\le p^{a_1}\le p^a$ and $\Xi_1\subset \Fp$ so that $\#A\cap(\Fp^{n-1}\times \xi)\sim p^{a_1}$ for any $\xi\in\Xi_1$ and
\[ p^{a_1}\cdot\#\Xi_1\gtrapprox p^a. \]
Since $\#\Xi_1\le p$, we also have $p^{a_1}\gtrapprox p^{a-1}$.

For each $\xi\in\Xi_1$, we consider the slice $\Fp^{n-1}\times \xi$. We denote $A_\xi:=A\cap (\Fp^{n-1}\times \xi)$. There exists a dyadic number $1\le p^{s_1(\xi)}\le p^{s-\e}$, such that
\[ \#\pi_W^*(A_\xi)\sim p^{s_1(\xi)}, \]
for $\gtrapprox \#\cW\sim p^{t_1}$ many $W\in \cW$. Here, $s_1(\xi)\le s-\e$ is because $\#\pi_W^*(A_\xi)\le \#\pi_V^*(A)<p^{s-\e}$ for $V\in\cV_W$.  

We see that $E_{s_1(\xi)+\e^2/2}(A_\xi;n-1,k)$ contains a $\gtrapprox 1$ portion of $\cW$ when $p$ is large enough depending on $\e$.
We can use \eqref{excondition} for the quadruple $(a_1,s_1(\xi)+\e^2;n-1,k)$ to obtain
\[ p^{t_1}\lessapprox \#E_{s_1(\xi)+\e^2/2}(A_\xi;n-1,k) \lesssim_\e p^{\e^2/2+\bM(a_1,s_1(\xi)+\e^2;n-1,k)}. \]

By pigeonholing on $s_1(\xi)$, we can find $\Xi\subset \Xi_1$ such that
\[ s_1(\xi)=s_1, \]
for all $\xi\in\Xi$. And, 
\[\#\Xi\gtrapprox\#\Xi_1\gtrapprox p^{a-a_1}.\]

Next, we note that
\[ \sum_{W\in\cW}\sum_{\xi\in\Xi}\#\pi_W^*(A_\xi)\gtrapprox \#\cW\cdot\#\Xi \cdot p^{s_1}. \]
By pigeonholing, there exists $W\in\cW$ such that
\[ \sum_{\xi\in\Xi}\#\pi_W^*(A_\xi)\gtrapprox \#\Xi\cdot p^{s_1}\gtrapprox p^{a-a_1+s_1}. \]
We fix this $W$ and study the behavior of $\cV_W$. Without loss of generality, we can assume $W=\Fp^{n-1-k}\times \bzz^{k}\times \bzz$. For $V\in \cV_W$, then $V\supset W$. We see that $V$ is of form $V=\Fp^{n-1-k}\times \ell_V$, where $\ell_V\in G(1,\bzz^{n-1-k}\times \Fp^{k+1})$. Since $V\cap \Fp^{n-1}\times \bzz=W$, we have that $\ell_V$ is not parallel to $\Fp^{n-1}\times \bzz$.

Next, we project everything onto $\bzz^{n-1-k}\times \Fp^{k+1}$. For $x=(x_1,\dots,x_n)\in\Fp^n$, we define 
\[\pi_{\Fp^{k+1}}(x):=(0,\dots,0,x_{n-k},x_{n-k+1},\dots,x_n).\]
One can also check that $\pi_W^*(x)=\pi_{\Fp^{k+1}}(x)+W$.

Define $B_\xi=\pi_{\Fp^{k+1}}(A_\xi)$. We see that $B_\xi$ is a subset of $\bzz^{n-1-k}\times \Fp^{k+1}$ and $\#B_\xi=\#\pi_W^*(A_\xi)$.

The next important observation is that for any $A\subset \Fp^n$, $V=\Fp^{n-1-k}\times \ell_V\in\cV_W$, one has
\[ \#\pi_{V}^*(A)=\#\pi_{\ell_V}^*(\pi_{\Fp^{k+1}}(A)). \]
This is not hard to verify once writing down the definition.

We can now do the estimate.
By definition, $\#\pi_V^*(A)<p^{s-\e}$ for $V\in\cV_W$. Therefore,
\[ p^{s-\e}> \#\pi_{\ell_V}^*(\pi_{\Fp^{k+1}}(A)). \]
This means that
\[ \{\ell_V: V\in\cV_W\}\subset E_{s-\e}(\pi_{\Fp^{k+1}}(A);k+1,k). \]
Also, note that 
\[ \#\pi_{\Fp^{k+1}}(A)\ge \#\pi_{\Fp^{k+1}}(\bigcup_{\xi\in\Xi}A_\xi)=\#\pi^*_{W}(\bigcup_{\xi\in\Xi}A_\xi)=\sum_{\xi\in\Xi}\#\pi_W^*(A_\xi)\gtrapprox p^{a-a_1+s_1}. \]
By \eqref{excondition} with $n'=k+1$, we get
\[ p^{t_2}\lessapprox \#\cV_W\le \#E_{s-\e}(\pi_{\Fp^{k+1}}(A);k+1,k)\lesssim_\e p^{\e/3+\bM(a-a_1+s_1-\e^2,s-\e/2;k+1,k)}.\]
(It is crucial that we use $s-\e/2$ instead of $s$ on the right hand side.)

Combining with the upper bound of $p^{t_1}$, we get
\[ \#\cV\lessapprox p^{t_1+t_2}\lesssim_\e p^{\e/2+\bM(a_1,s_1+\e^2;n-1,k)+\bM(a-a_1+s_1-\e^2,s-\e/2;k+1,k)}. \]

Let $s_1'=s_1+2\e^2$. Since $s_1\le s-\e$, we have $s_1'\le s$.

Recall $3\e^2<\e/2$ and note $\bM(a,s;n,k)$ is decreasing in $a$. By Lemma \ref{lipina}, we have 
\[\bM(a-a_1+s_1'-3\e^2,s-\e/2;k+1,k)\le \bM(a-a_1+s_1',s;k+1,k).\]
Therefore, under the new $s_1'$, we have
\[ \#\cV\lesssim_\e p^{\e+\bM(a_1,s_1'-\e^2;n-1,k)+\bM(a-a_1+s_1',s;k+1,k)}. \]
Next, we want to slightly increase $a_1$ to $a_1'$ so that $p^{\min\{0,a-1\}}\le p^{a_1'}$. By the condition $p^{\min\{0,a-1\}}\lessapprox p^{a_1}$, we have $p^{\min\{0,a-1\}}\le p^{a_1+\e^2}$ when $p$ is large enough. We just need to let $a_1'=\min\{a_1+\e^2, \min\{0,a-1\}\}$. Noting $a_1'\le a_1+\e^2$ and by Lemma \ref{lipina} again, we have
\[ \bM(a_1,s_1'-\e^2;n-1,k)\le \bM(a_1'-\e^2,s_1'-\e^2;n-1,k)\le \bM(a_1',s_1';n-1,k) \]
We still use $s_1,a_1$ for $s_1',a_1'$. We obtain
\[ \#\cV\lesssim_\e p^{\e+\bM(a_1,s_1;n-1,k)+\bM(a-a_1+s_1,s;k+1,k)}. \]

\end{proof}

\begin{lemma}\label{RecursionM}
     Suppose $n,k\in\N$, $n\ge k+2$, $a\in (0,n]$, $s>0$ and $s\in (a-(n-k),\min\{a,k\}]$. Suppose there exist numbers $a_1,s_1$, such that
    \begin{equation}
        \begin{cases}
            \max\{0,a-1\}\le a_1\le \min\{n-1,a\};\\
            0\le s_1\le s.
        \end{cases}
    \end{equation}
    Then
    \begin{equation}
         \bM(a_1,s_1;n-1,k)+\bM(s_1+a-a_1,s;k+1,k)\le \bM(a,s;n,k).
    \end{equation}
\end{lemma}
\begin{proof}
    Write $a=m+\beta,s=l+\ga$; $a_1=m_1+\beta_1, s_1=l_1+\ga_1$ in the canonical expressions. Since $a-1\le a_1\le a$, there are only two choices for $m_1$: $m_1=m$ or $m_1=m-1$. Since $s_1\le s$, we have $l_1\le l$.

    \medskip

    \fbox{Case 1: $m_1=m, l_1\le l-1$}

    By Lemma \ref{easyM},
    \[\bM(a_1,s_1;n-1,k)\le k(n-1-k)-(m-l+1)(k-l+1)+\max\{2\ga_1-(\beta_1+1),0\}.\]

    Trivially,
    \[ \bM(s_1+a-a_1,s;k+1,k)\le k. \]
    
    Also, by Lemma \ref{easyM},
    \[ \bM(a,s;n,k)\ge k(n-k)-(m+1-l)(k-l). \]

    It suffices to prove
    \[ \max\{2\ga_1-\beta_1-1,0\}\le m+1-l. \]
    We just need to note $m\ge l$ and $2\ga_1\le 2$.

\medskip

    \fbox{Case 2: $m_1=m, l_1=l$}

\medskip

    \fbox{Case 2.1: $\ga>\beta,\ga_1>\beta_1$}

    Since $s\le a$ and $\ga>\beta$, we have $m\ge l+1$. Hence,
    \begin{equation*}
        \begin{split}
            \bM(a_1,s_1;n-1,k)&=\bM(m+\beta_1,l+\ga_1;n-1,k)\\
            &=k(n-1-k)-(m-l)(k-l)+\max\{2\ga_1-(\beta_1+1),0\}. 
        \end{split}
    \end{equation*} 

    Next, we estimate $\bM(s_1+a-a_1,s;k+1,k)=\bM(l+\ga_1+\beta-\beta_1,l+\ga;k+1,k)$.
    We claim that we have
    \begin{equation}
        \bM(l+\ga_1+\beta-\beta_1,l+\ga;k+1,k)\le\begin{cases}
             k-(k-l)+\max\{2\ga-(\ga_1+\beta-\beta_1),0\} & \ga_1+\beta-\beta_1\ge \ga;\\
             k & \ga_1+\beta-\beta_1< \ga.
        \end{cases}
    \end{equation} 
    We just consider when $\ga_1+\beta-\beta_1\ge \ga$.
    We note $\ga_1+\beta-\beta_1\in (0,2)$. If $\ga_1+\beta-\beta_1\in(0,1]$, then we use Type 3 in Definition \ref{defM}, which gives the desired formula. If $\ga_1+\beta-\beta_1\in (1,2]$, we use Type 2, which gives $\bM(l+1+(\ga_1+\beta-\beta_1-1),l+\ga;k+1,k)= k-(k-l)+\max\{2\ga-(\ga_1+\beta-\beta_1),0\}$, which is the same. 

    We also have $\bM(a,s;n,k)=k(n-k)-(m-l)(k-l)+\max\{2\ga-(\beta+1),0\}$. 

    When $\ga_1+\beta-\beta_1<\ga$, we just need to prove
    \[ \max\{2\ga_1-(\beta_1+1),0\}\le \max\{2\ga-(\beta+1),0\}, \]
    which is easy to verify, since $\ga_1\le \ga$ and $\ga_1+\beta-\beta_1<\ga$.
    
     When $\ga_1+\beta-\beta_1\ge\ga$, we just need to prove
    \[ \max\{2\ga_1-(\beta_1+1),0\}+\max\{2\ga-(\ga_1+\beta-\beta_1),0\}\le 1+\max\{2\ga-(\beta+1),0\}. \]
    Just need to check three scenarios:

    (1) $2\ga_1-(\beta_1+1)\le 1$, since $\ga_1\le 1$.

    (2) $2\ga-(\ga_1+\beta-\beta_1)\le 1+2\ga-(\beta+1)$, since $\ga_1>\beta_1$.

    (3) Their sum $2\ga_1-(\beta_1+1)+ 2\ga-(\ga_1+\beta-\beta_1)=2\ga+\ga_1-(\beta+1)\le 1+2\ga-(\beta+1)$.

\medskip

\fbox{Case 2.2: $\ga>\beta, \ga_1\le \beta_1$}

In this case, by Lemma \ref{easyM},
\[ \bM(a_1,s_1;n-1,k)\le k(n-1-k)-(m+1-l)(k-l)+\max\{2\ga_1-\beta_1,0\}. \]

\[ \bM(l+\ga_1+\beta-\beta_1,l+\ga;k+1,k)\le k. \]

\[ \bM(a,s;n,k)=k(n-k)-(m-l)(k-l)+\max\{2\ga-(\beta+1),0\}. \]

It suffices to prove
\[ \max\{2\ga_1-\beta_1,0\}\le 1+\max\{2\ga-(\beta+1),0\}. \]
We just use $2\ga_1-\beta_1\le 1$, since $\ga_1\le \beta_1$.

\medskip

\fbox{Case 2.3: $\ga\le \beta,\ga_1>\beta_1$}

\[ \bM(a_1,s_1;n-1,k)\le k(n-1-k)-(m-l)(k-l)+\max\{2\ga_1-(\beta_1+1),0\}. \]
(We use $\le$ instead of $=$, since when $m=l$, $\bM(a_1,s_1;n-1,k)=k(n-1-k)$.)  

Since $\ga_1+\beta-\beta_1>\ga$,
\[ \bM(l+\ga_1+\beta-\beta_1,l+\ga;k+1,k)\le k-(k-l)+\max\{2\ga-(\ga_1+\beta-\beta_1),0\}. \]

Also, we have
\[ \bM(a,s;n,k)=k(n-k)-(m+1-l)(k-l)+\max\{2\ga-\beta,0\}. \]

It suffices to prove
\[ \max\{2\ga_1-(\beta_1+1),0\}+\max\{2\ga-(\ga_1+\beta-\beta_1),0\}\le \max\{2\ga-\beta,0\}. \]

(1) $2\ga_1-(\beta_1+1)\le 2\ga-\beta$, since $s_1\le s$ implies $\ga_1\le \ga$, and $\beta\le 1$.

(2) $2\ga-(\ga_1+\beta-\beta_1)\le 2\ga-\beta$, since $\ga_1>\beta_1$.

(3) $2\ga_1-(\beta_1+1)+2\ga-(\ga_1+\beta-\beta_1)=2\ga+\ga_1-(\beta+1)\le 2\ga-\beta$, since $\ga_1\le 1$.

\medskip

\fbox{Case 2.4: $\ga\le \beta, \ga_1\le \beta_1$}

\[ \bM(a_1,s_1;n-1,k)=k(n-1-k)-(m+1-l)(k-l)+\max\{2\ga_1-\beta_1,0\}. \]

\[ \bM(l+\ga_1+\beta-\beta_1,l+\ga;k+1,k)= \begin{cases}
    k & \ga>\ga_1+\beta-\beta_1\\
    k-(k-l)+\max\{2\ga-(\ga_1+\beta-\beta_1),0\} & \ga\le \ga_1+\beta-\beta_1.
\end{cases}  \]

\[ \bM(a,s;n,k)=k(n-k)-(m+1-l)(k-l)+\max\{2\ga-\beta,0\}. \] 

When $\ga>\ga_1+\beta-\beta_1$, it suffices to show
\[ \max\{2\ga_1-\beta_1,0\}\le \max\{2\ga-\beta,0\}. \]
This is true by noting $\ga_1\le \ga$ and $\ga>\ga_1+\beta-\beta_1$.

When $\ga\le \ga_1+\beta-\beta_1$, it suffices to show 
\[ \max\{2\ga_1-\beta_1,0\}+\max\{2\ga-(\ga_1+\beta-\beta_1),0\}\le 1+\max\{2\ga-\beta,0\}. \]

(1) $2\ga_1-\beta_1\le 1+2\ga-\beta$, since $\ga_1\le \ga$ and $\beta\le 1$.

(2) $2\ga-(\ga_1+\beta-\beta_1)\le 1+2\ga-\beta$, since $\beta_1\le 1$.

(3) $2\ga_1-\beta_1+2\ga-(\ga_1+\beta-\beta_1)=\ga_1-\beta+2\ga\le 1+2\ga-\beta$, since $\ga_1\le 1$.

\bigskip

\fbox{Case 3: $m_1=m-1, l_1\le l-2$}
The proof is similar to \fbox{Case 1}.
By Lemma \ref{easyM},
    \[\bM(a_1,s_1;n-1,k)\le k(n-1-k)-(m-l)(k-l+2)+\max\{2\ga_1-(\beta_1+1),0\}.\]

    Trivially,
    \[ \bM(s_1+a-a_1,s;k+1,k)\le k. \]
    
    Also, by Lemma \ref{easyM},
    \[ \bM(a,s;n,k)\ge k(n-k)-(m-l)(k-l+1). \]

    It suffices to prove
    \[ \max\{2\ga_1-\beta_1-1,0\}\le m-l. \]
    We just need to note $m-l\ge 1$ and $2\ga_1\le 2$.

    \medskip

    \fbox{Case 4: $m_1=m-1,l_1=l-1$}
    \medskip

     \fbox{Case 4.1: $\ga>\beta,\ga_1>\beta_1$}

    Since $s\le a$ and $\ga>\beta$, we have $m\ge l+1$. Hence,
    \begin{equation*}
        \begin{split}
            \bM(a_1,s_1;n-1,k)&=\bM(m-1+\beta_1,l-1+\ga_1;n-1,k)\\
            &=k(n-1-k)-(m-l)(k-l+1)+\max\{2\ga_1-(\beta_1+1),0\}. 
        \end{split}
    \end{equation*} 

    Next, we estimate $\bM(s_1+a-a_1,s;k+1,k)=\bM(l+\ga_1+\beta-\beta_1,l+\ga;k+1,k)$.
   As is \fbox{Case 2.1},
    \begin{equation}
        \bM(l+\ga_1+\beta-\beta_1,l+\ga;k+1,k)\le\begin{cases}
             k-(k-l)+\max\{2\ga-(\ga_1+\beta-\beta_1),0\} & \ga_1+\beta-\beta_1\ge \ga;\\
             k & \ga_1+\beta-\beta_1< \ga.
        \end{cases}
    \end{equation}

    We also have $\bM(a,s;n,k)=k(n-k)-(m-l)(k-l)+\max\{2\ga-(\beta+1),0\}$. 

    When $\ga_1+\beta-\beta_1<\ga$, we only need to prove
    \[ \max\{2\ga_1-(\beta_1+1),0\}\le 1+\max\{2\ga-(\beta+1),0\}, \]
    which is easy to verify,since $\ga_1\le 1$ and $\ga_1+\beta-\beta_1<\ga$.
    
     When $\ga_1+\beta-\beta_1\ge\ga$, we only need to prove
    \[ \max\{2\ga_1-(\beta_1+1),0\}+\max\{2\ga-(\ga_1+\beta-\beta_1),0\}\le 2+\max\{2\ga-(\beta+1),0\}. \]
    We just need to check three scenarios:

    (1) $2\ga_1-(\beta_1+1)\le 2$, since $\ga_1\le 1$.

    (2) $2\ga-(\ga_1+\beta-\beta_1)\le 2$, since $\ga_1+\beta-\beta_1\ge\ga \ge 0$.

    (3) $2\ga_1-(\beta_1+1)+2\ga-(\ga_1+\beta-\beta_1)=2\ga+\ga_1-(\beta+1)\le 2+2\ga-(\beta+1)$, since $\ga_1\le 1$.

\medskip

\fbox{Case 4.2: $\ga>\beta, \ga_1\le \beta_1$}

In this case,
\[ \bM(a_1,s_1;n-1,k)\le k(n-1-k)-(m+1-l)(k-l+1)+\max\{2\ga_1-\beta_1,0\}. \]

\[ \bM(l+\ga_1+\beta-\beta_1,l+\ga;k+1,k)\le k. \]

\[ \bM(a,s;n,k)=k(n-k)-(m-l)(k-l)+\max\{2\ga-(\beta+1),0\}. \]

It suffices to prove
\[ \max\{2\ga_1-\beta_1,0\}\le 1+\max\{2\ga-(\beta+1),0\}. \]
We just use $2\ga_1-\beta_1\le 1$, since $\ga_1\le \beta_1$.

\medskip

\fbox{Case 4.3: $\ga\le \beta,\ga_1>\beta_1$}

\[ \bM(a_1,s_1;n-1,k)\le \begin{cases}
    k(n-1-k)-(m-l)(k-l+1)+\max\{2\ga_1-(\beta_1+1),0\} & m>l\\
    k(n-1-k) & m=l.
\end{cases} \]

Since $\ga_1+\beta-\beta_1>\ga$,
\[ \bM(l+\ga_1+\beta-\beta_1,l+\ga;k+1,k)\le k-(k-l)+\max\{2\ga-(\ga_1+\beta-\beta_1),0\}. \] Also, we can assume $\ga_1+\beta-\beta_1\le \ga+1$; otherwise, it $=-\infty$.

Also, we have
\[ \bM(a,s;n,k)=k(n-k)-(m+1-l)(k-l)+\max\{2\ga-\beta,0\}. \]

When $m=l$, we only need to prove
\[ \max\{2\ga-(\ga_1+\beta-\beta_1),0\}\le \max\{2\ga-\beta,0\}. \]
It is easy to verify $2\ga-(\ga_1+\beta-\beta_1)\le 2\ga-\beta$.

When $m>l$, it suffices to prove
\[ \max\{2\ga_1-(\beta_1+1),0\}+\max\{2\ga-(\ga_1+\beta-\beta_1),0\}\le 1+ \max\{2\ga-\beta,0\}. \]

(1) $2\ga_1-(\beta_1+1)\le 1+ 2\ga-\beta$, since $\ga_1+\beta-\beta_1\le \ga+1$ and $\ga_1\le 1$.

(2) $2\ga-(\ga_1+\beta-\beta_1)\le 1+2\ga-\beta$, since $\ga_1>\beta_1$.

(3) $2\ga_1-(\beta_1+1)+2\ga-(\ga_1+\beta-\beta_1)=2\ga+\ga_1-(\beta+1)\le 2\ga-\beta$, since $\ga_1\le 1$.

\medskip

\fbox{Case 4.4: $\ga\le \beta, \ga_1\le \beta_1$}

\[ \bM(a_1,s_1;n-1,k)=k(n-1-k)-(m+1-l)(k-l+1)+\max\{2\ga_1-\beta_1,0\}. \]

\[ \bM(l+\ga_1+\beta-\beta_1,l+\ga;k+1,k)\le \begin{cases}
    k & \ga>\ga_1+\beta-\beta_1\\
    k-(k-l)+\max\{2\ga-(\ga_1+\beta-\beta_1),0\} & \ga\le \ga_1+\beta-\beta_1.
\end{cases}  \]
We can also assume $\ga_1+\beta-\beta_1\le \ga+1$; otherwise, it $=-\infty$.

\[ \bM(a,s;n,k)=k(n-k)-(m+1-l)(k-l)+\max\{2\ga-\beta,0\}. \] 

When $\ga>\ga_1+\beta-\beta_1$, it suffices to show
\[ \max\{2\ga_1-\beta_1,0\}\le 1+ \max\{2\ga-\beta,0\}. \]
This is true by noting $\ga_1\le 1$ and $\ga>\ga_1+\beta-\beta_1$.

When $\ga\le \ga_1+\beta-\beta_1$, it suffices to show 
\[ \max\{2\ga_1-\beta_1,0\}+\max\{2\ga-(\ga_1+\beta-\beta_1),0\}\le 2+\max\{2\ga-\beta,0\}. \]

(1) $2\ga_1-\beta_1\le 2+2\ga-\beta$, since $\ga_1+\beta-\beta_1\le \ga+1$, $\ga_1\le 1$.

(2) $2\ga-(\ga_1+\beta-\beta_1)\le 2+2\ga-\beta$, since $\beta_1\le 1$.

(3) $2\ga_1-\beta_1+2\ga-(\ga_1+\beta-\beta_1)=\ga_1-\beta+2\ga\le 2+2\ga-\beta$, since $\ga_1\le 1$.
    
\medskip

\fbox{Case 5: $m_1=m-1, l_1=l$}

\medskip

    \fbox{Case 5.1: $\ga>\beta,\ga_1>\beta_1$}

    In this case, we have
    \[ \bM(m-1+\beta_1,l+\ga_1;n-1,k)\le k(n-1-k)-(m-1-l)(k-l)+\max\{2\ga_1-(\beta_1+1),0\}. \]

    \[ \bM(l+1+\ga_1+\beta-\beta_1,l+\ga;k+1,k)\le k-(k-l)+\max\{2\ga-(1+\ga_1+\beta-\beta_1),0\}. \]
    We assume \begin{equation}
        \ga>\ga_1+\beta-\beta_1.
    \end{equation} Otherwise, this is $-\infty$.

    We also have $\bM(a,s;n,k)=k(n-k)-(m-l)(k-l)+\max\{2\ga-(\beta+1),0\}$. We just need to prove
    \[ \max\{2\ga_1-(\beta_1+1),0\}+\max\{2\ga-(1+\ga_1+\beta-\beta_1),0\}\le \max\{2\ga-(\beta+1),0\}. \]
    Check three scenarios:

    (1) $2\ga_1-(\beta_1+1)\le 2\ga-(\beta+1)$, since $\ga>\ga_1+\beta-\beta_1$ and $\ga_1\le \ga$ (since $s_1\le s$).

    (2) $2\ga-(1+\ga_1+\beta-\beta_1)\le 2\ga-(\beta+1)$, since $\ga_1>\beta_1$.

    (3) $2\ga_1-(\beta_1+1)+2\ga-(1+\ga_1+\beta-\beta_1)=2\ga+\ga_1-(\beta+2)\le 2\ga-(\beta+1)$, since $\ga_1\le 1$.

\medskip

\fbox{Case 5.2: $\ga>\beta, \ga_1\le \beta_1$}

In this case, by discussing $m=l$ or $m>l$, we have
\[ \bM(a_1,s_1;n-1,k)\le k(n-1-k)-(m-l)(k-l)+\max\{2\ga_1-\beta_1,0\}. \]

\[ \bM(l+1+\ga_1+\beta-\beta_1,l+\ga;k+1,k)\le k-(k-l)+\max\{2\ga-(1+\ga_1+\beta-\beta_1),0\}. \]

\[ \bM(a,s;n,k)=k(n-k)-(m-l)(k-l)+\max\{2\ga-(\beta+1),0\}. \]

It suffices to prove
\[ \max\{2\ga_1-\beta_1,0\}+\max\{2\ga-(1+\ga_1+\beta-\beta_1),0\}\le 1+\max\{2\ga-(\beta+1),0\}. \]

(1) $2\ga_1-\beta_1\le 1$, since $\ga_1\le \beta_1$.

(2) $2\ga-(1+\ga_1+\beta-\beta_1)\le 2\ga-\beta$, since $\beta_1\le 1$.

(3) $2\ga+\ga_1-(1+\beta)\le 2\ga-\beta$, since $\ga_1\le 1$.

\medskip

\fbox{Case 5.3: $\ga\le \beta,\ga_1>\beta_1$}

\[ \bM(l+1+\ga_1+\beta-\beta_1,l+\ga;k+1,k)=-\infty. \]
There is nothing to do.

\medskip

\fbox{Case 5.4: $\ga\le \beta, \ga_1\le \beta_1$}

\[ \bM(a_1,s_1;n-1,k)\le k(n-1-k)-(m-l)(k-l)+\max\{2\ga_1-\beta_1,0\}. \]

\[ \bM(l+1+\ga_1+\beta-\beta_1,l+\ga;k+1,k)\le 
    k-(k-l)+\max\{2\ga-(1+\ga_1+\beta-\beta_1),0\}.  \]
We assume
\begin{equation}
\ga>\ga_1+\beta-\beta_1.
\end{equation}
Otherwise, this is $-\infty.$

\[ \bM(a,s;n,k)=k(n-k)-(m+1-l)(k-l)+\max\{2\ga-\beta,0\}. \] 

It suffices to show 
\[ \max\{2\ga_1-\beta_1,0\}+\max\{2\ga-(1+\ga_1+\beta-\beta_1),0\}\le \max\{2\ga-\beta,0\}. \]

(1) $2\ga_1-\beta_1\le 2\ga-\beta$, since $\ga>\ga_1+\beta-\beta_1$ and $\ga_1\le \ga$.

(2) $2\ga-(1+\ga_1+\beta-\beta_1)\le 2\ga-\beta$, since $\beta_1\le 1$.

(3) $2\ga+\ga_1-\beta-1\le 2\ga-\beta$, since $\ga_1\le 1$.
\end{proof}

\begin{proof}[Proof of Proposition \ref{expropupper2}]
    We just combine Lemma \ref{leminducM} and Lemma \ref{RecursionM}.
\end{proof}

\begin{proof}[Proof of Theorem \ref{exceptionalthm}]
    Finally, we see that Theorem \ref{exceptionalthm} can be obtained from Proposition \ref{exprop1} and Proposition \ref{expropupper2} by induction on $(n,k)$.
\end{proof}

\section*{Acknowledgments} 
The author is grateful to the anonymous reviewers for finding
typos in the paper.

\bibliographystyle{amsplain}


\begin{dajauthors}
\begin{authorinfo}[sgan]
  Shengwen Gan\\
Department of Mathematics\\
 University of Wisconsin-Madison\\
 Madison, WI-53706, USA
   \\
   \url{https://sites.google.com/view/shengwengan/home}
\end{authorinfo}

\end{dajauthors}

\end{document}